\documentclass[12pt, leqno]{amsart}

\usepackage{amssymb}

\usepackage[all]{xy}

\theoremstyle{plain}
\newtheorem*{thm}{Theorem}
\newtheorem{theorem}{Theorem}[section]
\newtheorem{lemma}[theorem]{Lemma}
\newtheorem{proposition}[theorem]{Proposition}

\theoremstyle{definition}
\newtheorem{definition}[theorem]{Definition}
\newtheorem{remark}[theorem]{Remark}

\newtheorem{example}[theorem]{Example}

\newcommand{\ldash}{\text -}

\newcommand\bD{{\mathbb D}}

\newcommand\bG{{\mathbb G}}

\newcommand\bP{{\mathbb P}}
\newcommand\bQ{{\mathbb Q}}

\newcommand\bZ{{\mathbb Z}}

\newcommand\cA{{\mathcal A}}
\newcommand\cB{{\mathcal B}}
\newcommand\cC{{\mathcal C}}

\newcommand\cE{{\mathcal E}}
\newcommand\cF{{\mathcal F}}

\newcommand\cI{{\mathcal I}}
\newcommand\cJ{{\mathcal J}}
\newcommand\cL{{\mathcal L}}
\newcommand\cM{{\mathcal M}}

\newcommand\cO{{\mathcal O}}

\newcommand\cU{{\mathcal U}}

\newcommand\ucC{\underline{\mathcal C}}

\newcommand\ucS{\underline{\mathcal S}}

\newcommand\ucU{\underline{\mathcal U}}

\newcommand\cd{{\rm cd}}
\newcommand\charac{{\rm char}}

\newcommand\cont{{\rm cont}}

\newcommand\hd{{\rm hd}}
\newcommand\id{{\rm id}}

\renewcommand\mod{{\rm mod}}

\newcommand\red{{\rm red}}

\newcommand\tors{{\rm tors}}

\newcommand\Ac{{\rm Ac}}

\newcommand\Coker{{\rm Coker}}

\newcommand\End{{\rm End}}
\newcommand\Ext{{\rm Ext}}

\newcommand\Hom{{\rm Hom}}
\renewcommand\Im{{\rm Im}}
\newcommand\Ind{{\rm Ind}}

\newcommand\Ker{{\rm Ker}}

\newcommand\Mod{{\rm Mod}}

\newcommand\Pro{{\rm Pro}}

\numberwithin{equation}{subsection}

\title[Homological dimension of isogeny categories]
{Homological dimension of isogeny categories\\
of commutative algebraic groups}

\author{Michel Brion}

\date{}

\begin{document}

\begin{abstract}
We determine the homological dimension of various isogeny categories 
of commutative algebraic groups over a field $k$, in terms of the 
cohomological dimension of $k$ at certain primes. This generalizes
results of Serre, Oort and Milne, by an alternative approach.
\end{abstract}

\maketitle

\section{Introduction}
\label{sec:int}

The commutative group schemes of finite type over a field $k$
form an abelian category $\cC$. When $k$ is algebraically closed, 
the homological properties of $\cC$ have been investigated by 
Serre and Oort (see \cite{Serre, Oort66}); in particular, 
the homological dimension turns out to be $1$ in characteristic $0$, 
and $2$ otherwise. Building on these results, Milne determined the
homological dimension of $\cC$ when $k$ is perfect: then
\[ \hd(\cC) = \begin{cases} \cd(k) + 1 & \text{if $\charac(k) = 0$,} \\
\max(2,\cd(k) + 1)  & \text{otherwise}
\end{cases} \]
(see \cite{Milne}). Here $\cd(k)$ denotes the cohomological dimension 
of $k$, i.e., the supremum of the $\ell$th cohomological dimensions 
$\cd_{\ell}(k)$ over all primes $\ell$. 

The aim of this paper is to generalize these results to various
isogeny categories, obtained from $\cC$ by formally inverting 
isogenies of a certain type; these may also be viewed as quotient
categories $\cC/\cB$, where $\cB$ is a Serre subcategory of $\cC$
all of whose objects are finite group schemes. The first result in this 
direction is again due to Serre in \cite{Serre}: if $k$ is algebraically 
closed of positive characteristic, then $\hd(\cC/\cI) = 2$, where $\cI$ 
denotes the full subcategory of $\cC$ with objects the infinitesimal 
group schemes. Note that $\cC/\cI$ is equivalent to the category of
``quasi-algebraic groups'' considered in \cite{Serre}; it is obtained
from $\cC$ by inverting the radicial (or purely inseparable) isogenies.

More recently, it was shown in \cite{Brion-II} that 
$\hd(\cC/\cF) =1$ for an arbitrary field $k$, where $\cF$ denotes
the Serre subcategory of $\cC$ with objects the finite group schemes.
For this, we adapted the original approach of Serre and Oort, in the 
somewhat simpler setting of the full isogeny category $\cC/\cF$.
A more conceptual proof is presented in \cite{Brion-III}, which also 
contains a representation-theoretic description of $\cC/\cF$.

In this paper, we consider the radicial isogeny category $\cC/\cI$
over an arbitrary field, as well as the $S$-isogeny category
$\cC/\cF_S$ and the \'etale $S$-isogeny category $\cC/\cE_S$, 
where $S$ denotes a set of prime numbers, and $\cF_S$ (resp.~$\cE_S$) 
is the full subcategory of $\cF$ with objects the $S$-primary torsion 
groups (resp.~the \'etale $S$-primary torsion groups). 
Our main result can be formulated as follows:

\begin{thm}
Let $k$ be a field of characteristic $p \geq 0$, and 
$S$ a subset of the set $\bP$ of prime numbers. Set 
$\cd_{S'}(k) := \sup_{\ell \notin S} (\cd_{\ell}(k))$
if $S \neq \bP$, and $\cd_{S'}(k) = 0$ if $S = \bP$.

\begin{enumerate}

\item[{\rm (i)}] If $p = 0$ or $p \in S$, then
$\hd(\cC/\cF_S) = \cd_{S'}(k) + 1$.

\item[{\rm (ii)}] If $k$ is perfect and $p > 0$, then
$\hd(\cC/\cE_S) = \max(2,\cd_{S'}(k) + 1)$.

\item[{\rm (iii)}] If $p > 0$, then 
$\hd(\cC/\cI) = \max(2,\cd(k) + 1)$.

\end{enumerate}

\end{thm}

This yields the values of $\hd(\cC/\cF_S)$ and 
$\hd(\cC/\cE_S)$ for an arbitrary set $S$ when $k$ is perfect, 
since $\cE_S = \cF_S$  if $p \notin S$. 
Also, this gives back Milne's formula: just take $S = \emptyset$. 
But the determination of $\hd(\cC)$ when $k$ is imperfect seems
to be an open question, see Remark \ref{rem:Fhd} for more details.

Our approach is independent from the results of \cite{Serre,Oort66} 
(which describe all higher extension groups for ``elementary'' 
algebraic groups) and \cite{Milne} (which constructs a spectral 
sequence relating the higher extension groups over the field $k$ 
with those over its algebraic closure, via the cohomology of 
the absolute Galois group). We follow a suggestion of Oort 
(see \cite[II.15.2]{Oort66}), albeit in a slightly different setting
which yields a more uniform result. Specifically, we show that 
\begin{equation}\label{eqn:tors}
\hd(\cA) = \max(\hd(\cA_{\tors}), \hd(\cA/\cA_{\tors})) 
\end{equation}
for any artinian abelian category $\cA$, where $\cA_{\tors}$
denotes the full subcategory of $\cA$ consisting of those objects
$X$ such that the multiplication map $n_X$ is zero for some
positive integer $n = n(X)$. 
When $\cA$ is one of the isogeny categories considered above 
(which are indeed artinian), it turns out that $\cA/\cA_{\tors}$ 
has homological dimension at most $1$, and $\cA_{\tors}$ can be 
explicitly determined in each case. 

This approach also applies to the full subcategory $\cL$ of $\cC$ 
with objects the linear (or equivalently, affine) group schemes.
Note that $\cL$ contains $\cF$; thus, we may consider the
isogeny category $\cL/\cB$, where $\cB$ equals $\cF_S$, $\cE_S$
or $\cI$. It turns out that $\hd(\cL/\cB) = \hd(\cC/\cB)$ in all cases,
see Remarks \ref{rem:Fhd}, \ref{rem:Ihd} and \ref{rem:pi} for details.

The layout of this paper is as follows. In Section \ref{sec:prel},
we gather preliminary results on artinian abelian categories and
the corresponding pro-artinian categories. These results are 
closely related to classical work of Gabriel in \cite[Chap.~III]{Gabriel}. 
Some of them may be well known (for example, the fact that the quotient 
of any artinian abelian category by a Serre subcategory is artinian), 
but we could not locate appropriate references. The main result of this
section is Theorem \ref{thm:hd}, which yields a more general version
of (\ref{eqn:tors}).

In Section \ref{sec:isogeny}, we first describe the three subcategories 
$\cF_S$, $\cE_S$ and $\cI$ of $\cC$, and the corresponding isogeny 
categories. Then we prove the main theorem under the assumption 
that $k$ is perfect, for each of these three cases 
(Theorems \ref{thm:Fhd}, \ref{thm:Ihd} and \ref{thm:Ehd}).
Finally, we show that the categories $\cC/\cI$ and $\cC/\cF_S$,
where $p \in S$, are invariant under purely inseparable extension
of the base field $k$, thereby completing the proof of the main theorem.    

Further developments and applications to the fundamental group are
presented in \cite{Brion-IV}.

\section{Preliminaries on artinian abelian categories}
\label{sec:prel}

\subsection{Quotients of artinian categories}
\label{subsec:quotients}

Throughout this subsection, $\cA$ denotes an abelian category. We assume
that $\cA$ is \emph{artinian}, i.e., for any $X \in \cA$, every descending
chain of subobjects of $X$ is stationary.

Let $\cB$ be a full subcategory of $\cA$. We assume that $\cB$ is 
\emph{a Serre subcategory}, i.e., is stable under taking subobjects,
quotients and extensions in $\cA$; then $\cB$ is an abelian subcategory
of $\cA$. We denote by ${^{\perp}\cB}$ the full subcategory of $\cA$ 
consisting of those objects $X$ for which $\Hom_{\cA}(X,Y) = 0$ 
for all $Y \in \cB$; thus, $\cB$ and ${^{\perp}\cB}$ are disjoint.

\begin{lemma}\label{lem:exact}

\begin{enumerate}

\item[{\rm (i)}] The subcategory ${^{\perp}\cB}$ is stable under 
taking quotients and extensions in $\cA$.

\item[{\rm (ii)}] For any $X \in \cA$, there exists an exact
sequence
\begin{equation}\label{eqn:ext} 
0 \longrightarrow X^{\cB} \longrightarrow X \longrightarrow 
X_{\cB} \longrightarrow 0, 
\end{equation}
where $X^{\cB} \in {^{\perp}\cB}$ and $X_{\cB} \in \cB$,
and such a sequence is unique up to isomorphism.

\end{enumerate}

\end{lemma}

\begin{proof}
(i) This follows readily from the fact that every exact sequence in $\cA$
\[ 0 \longrightarrow X_1 \longrightarrow X \longrightarrow X_2
\longrightarrow 0 \]
yields an exact sequence
\[ 0 \longrightarrow \Hom_{\cA}(X_2,Y) \longrightarrow
\Hom_{\cA}(X,Y) \longrightarrow \Hom_{\cA}(X_1,Y). \]

(ii) Since $\cA$ is artinian, there exists a subobject 
$Z$ of $X$ such that $X/Z \in \cB$ and $Z$ is minimal for
this property. Assume that $Z \notin {^{\perp}\cB}$; then we may 
choose a non-zero morphism $f: Z \to W$, where $W \in \cB$. Since
$\cB$ is a Serre subcategory of $\cA$, we have that
$Z/\Ker(f) \in \cB$ and then $X/\Ker(f) \in \cB$. But 
$\Ker(f) \subsetneq Z$, a contradiction. This shows the existence
of the exact sequence (\ref{eqn:ext}).

For the uniqueness, consider another exact sequence
\[ 0 \longrightarrow Z' \longrightarrow X \longrightarrow Y' 
\longrightarrow 0, \]
where $Z' \in {^{\perp}\cB}$ and $Y' \in \cB$. View $Z$ and $Z'$ as
subobjects of $X$. By the definition of ${^{\perp}\cB}$, 
the composition $Z' \to X \to Y$ is zero; thus, $Z'$ is a subobject
of $Z$. Then $Z/Z'$ is in ${^{\perp}\cB}$ (as a quotient of 
$Z \in {^{\perp}\cB}$) and in $\cB$ (as a subobject of $X/Z' = Y'$),
and hence is zero. Thus, $Z' = Z$, and $Y' \cong Y$.
\end{proof}

\begin{remark}\label{rem:torsion}
By Lemma \ref{lem:exact} (ii), the pair $({^{\perp}\cB}, \cB)$
of subcategories of $\cA$ is a torsion pair, as defined e.g.~in
\cite[Sec.~1.1]{Beligiannis-Reiten}. As a consequence, the 
assignment $X \mapsto X_{\cB}$ extends to an additive functor
$\cA \to \cB$ which is left adjoint to the inclusion 
$\cB \to \cA$, and hence right exact. Likewise, the 
assignment $X \mapsto X^{\cB}$ extends to an additive functor
$\cA \to {^{\perp}\cB}$ which is right adjoint to the inclusion 
${^{\perp}\cB} \to \cA$, and hence left exact (see [loc.~cit.] for
these results).
\end{remark} 

Next, consider the quotient category $\cA/\cB$: its objects
are those of $\cA$ and we have 
\[ \Hom_{\cA/\cB}(X,Y) = 
\lim_{\to} \big( \Hom_{\cA}(X',Y/Y') \big) \]
for all $X,Y \in \cA$, where the direct limit is taken
over all subobjects $X' \subset X$ such that $X/X' \in \cB$,
and all subobjects $Y'\subset Y$ such that $Y' \in \cB$.
We denote by 
\[ Q : \cA \longrightarrow \cA/\cB \]
the quotient functor (which is exact). 
In particular, we have a canonical morphism of abelian groups
\[ Q = Q(X,Y) : \Hom_{\cA}(X,Y) \longrightarrow \Hom_{\cA/\cB}(X,Y) \]
for any $X,Y \in \cA$.  

We now obtain a simpler description of the quotient category
$\cA/\cB$, which generalizes \cite[Lem.~3.1]{Brion-II}:

\begin{lemma}\label{lem:equiv}
Let $\cC$ be the full subcategory of $\cA/\cB$ with objects
those of ${^{\perp}\cB}$.

\begin{enumerate}

\item[{\rm (i)}] The inclusion of $\cC$ in $\cA/\cB$
is an equivalence of categories.

\item[{\rm (ii)}] For any $X,Y \in \cC$, we have
\[ \Hom_{\cC}(X,Y) = \lim_{\to} \big( \Hom_{\cA}(X,Y/Y') \big), \]
where the direct limit is taken over all $Y' \subset Y$ such that
$Y' \in \cB$.

\item[{\rm (iii)}] With the assumptions of {\rm (ii)}, let 
$\varphi \in \Hom_{\cC}(X,Y)$ with representative 
$f \in \Hom_{\cA}(X,Y/Y')$. Then $\varphi = 0$ 
(resp.~$\varphi$ is a mono\-morphism, an epimorphism)
if and only if $f = 0$ (resp.~$\Ker(f) \in \cB$, 
$f$ is an epimorphism).

\end{enumerate}
 
\end{lemma}
 
\begin{proof}
(i) Let $X \in \cA$, then the morphism $X^{\cB} \to X$ is an isomorphism
in $\cA/\cB$ with the notation of Lemma \ref{lem:exact} (ii). Thus,
the inclusion $\cC \to \cA/\cB$ is essentially surjective.

(ii) Let $X \in {^{\perp}\cB}$ and $X' \subset X$ such that $X/X' \in \cB$.
Then the quotient morphism $X \to X/X'$ is zero by the definition of
${^{\perp}\cB}$. Thus, $X' = X$.

(iii) By \cite[III.1.Lem.~2]{Gabriel}, $\varphi$ is zero
(resp.~a monomorphism, an epimorphism) if and only if 
$\Im(f)$ (resp.~$\Ker(f)$, $\Coker(f)$) is in $\cB$. 
As $X \in {^{\perp}\cB}$, we have $\Im(f) \in \cB$ if and only if  
$f = 0$. Likewise, as $Y \in {^{\perp}\cB}$, we have $\Coker(f) \in \cB$
if and only if $f$ is an epimorphism.
\end{proof}

\begin{lemma}\label{lem:mono}
Let $X,Y \in \cA$ and let $\varphi \in \Hom_{\cA/\cB}(X,Y)$
be a monomorphism. Then there exists a commutative 
triangle in $\cA/\cB$
\[
\xymatrix{
\tilde{X} \ar[dr]^-{Q(f)} \ar[d]_{\psi} \\
X \ar[r]^{\varphi} & Y, \\
}
\]
where $\psi$ is an isomorphism and $f \in \Hom_{\cA}(\tilde{X},Y)$ 
is a monomorphism.
\end{lemma}

\begin{proof}
By Lemma \ref{lem:equiv}, we may assume that $X \in {^{\perp}\cB}$;
then $\varphi$ is represented by a morphism
$g : X \to Y/Y'$ in $\cA$, where $Y' \subset Y$ and $Y' \in \cB$.
Moreover, $\Ker(g) \in \cB$ as $\varphi$ is a monomorphism in $\cA/\cB$.
We may thus replace $X$ by $X/\Ker(g)$, and assume
that $g$ is a monomorphism in $\cA$. Form the fibered square in $\cA$
\[ \xymatrix{
\tilde{X} \ar[r]^{f} \ar[d]_{\tilde{q}}
&  Y \ar[d]^{q} \\
X \ar[r]^-{g} & Y/Y',\\
}
\] 
where $q$ denotes the quotient morphism. Then $f$ is a monomorphism 
in $\cA$; also, $q$ and $\tilde{q}$ yield isomorphisms in $\cA/\cB$. 
\end{proof}

\begin{proposition}\label{prop:artin}
The quotient category $\cA/\cB$ is artinian.
\end{proposition}

\begin{proof}
Let $X \in \cA$ and consider a descending chain
$(X_n)_{n \geq 1}$ of subobjects of $X$ in $\cA/\cB$,
i.e., $X_1 = X$ and there exist monomorphisms 
$\varphi_n \in \Hom_{\cA/\cB}(X_{n+1},X_n)$ for all $n \geq 1$. 
In view of Lemma \ref{lem:mono}, there is a monomorphism
$f_1: \tilde{X}_2 \to X_1$ in $\cA$ such that
$\tilde{X}_2 \cong X_2$ in $\cA/\cB$ and this isomorphism
identifies $Q(f_1)$ with $\varphi_1$. Thus, we may replace
$\varphi_1$ with $Q(f_1)$ and assume that $\varphi_1$ is 
the image of a monomorphism in $\cA$. Iterating this construction, 
we may assume that $(X_n)_{n \geq 1}$ is a descending chain of
subobjects of $X$ in $\cA$. Hence this sequence is stationary.  
\end{proof}

\subsection{The lifting property}
\label{subsec:lifting}

We continue to consider an abelian category $\cA$ and a Serre 
subcategory $\cB$.

\begin{definition}\label{def:lift}
The pair $(\cA,\cB$) satisfies the \emph{lifting property} 
if for any epimorphism $X \to Y$ in $\cA$, where $Y \in \cB$, 
there exists a subobject $X' \subset X$ such that $X' \in \cB$ 
and the composition $X' \to X \to Y$ is an epimorphism in $\cB$.

Equivalently, given an exact sequence 
\[ 0 \longrightarrow Z \longrightarrow X \longrightarrow Y 
\longrightarrow 0 \] 
in $\cA$, where $Y \in \cB$, there exists a commutative diagram
of exact sequences 
\[ \xymatrix{
0 \ar[r] & W \ar[r] \ar[d] & V \ar[r] \ar[d] & Y \ar[r] \ar[d]^{\id} & 0 \\
0 \ar[r] & Z \ar[r] & X \ar[r] & Y \ar[r] & 0,
}
\]
where $V, W \in \cB$.
\end{definition}

The lifting property is often used to show that the bounded derived 
category of $\cB$ is naturally equivalent to the full subcategory of
the bounded derived category of $\cA$ with objects having
cohomology in $\cB$ (see \cite[Thm.~13.2.8]{KS} and its applications). 
We now obtain a handy description of morphisms in the quotient category 
$\cA/\cB$:

\begin{lemma}\label{lem:homq}
Assume that $(\cA,\cB)$ satisfies the lifting property. Then
the natural map
\[ \lim_{\to} \big( \Hom_{\cA}(X,Y/Y') \big) \longrightarrow 
\Hom_{\cA/\cB}(X,Y) \]
is an isomorphism for all $X,Y \in \cA$ (where the direct limit 
is taken over all $Y' \subset Y$ such that $Y' \in \cB$). 
\end{lemma}

\begin{proof}
Let $\varphi \in \Hom_{\cA/\cB}(X,Y)$ be represented by
$f \in \Hom_{\cA}(X',Y/Y')$, where $X' \subset X$, $Y' \subset Y$ 
and $X/X',Y' \in \cB$. By the lifting property, there exists
$Z \subset X$ such that $Z \in \cB$ and $X = X' + Z$. 
We then have a commutative diagram in $\cA$
\[ 
\xymatrix{
X \ar@{->>}[d] & X' \ar@{_{(}->}[l]_{i} \ar[r]^{f} \ar@{->>}[d] 
& Y/Y' \ar@{->>}[d] \\
X/Z & X'/X' \cap Z \ar[l]_-{j} \ar[r]^-{g} & Y/Y'',\\
}
\]
where $i$ denotes the inclusion, $j$ the canonical isomorphism,
$X' \cap Z \in \cB$, $Y'' := Y' + f(X' \cap Z) \in \cB$, 
and the vertical arrows are quotients.
It follows that $\varphi$ is represented by a morphism 
$h : X \to Y/Y''$ in $\cA$. This shows the surjectivity of 
the considered map.

For the injectivity, let $f \in \Hom_{\cA}(X,Y/Y')$ such
that $Q(f) = 0$. Then $\Im(f) \in \cB$ by 
\cite[III.1.Lem.~2]{Gabriel}. Denote by $Y''$ the preimage
of $\Im(f)$ in $Y$; then $Y'' \in \cB$ 
and the composition $X \to Y/Y' \to Y/Y''$ is zero in $\cA$.
\end{proof}

The following lemma yields a dual statement to that of Lemma 
\ref{lem:exact}, in the presence of the lifting property:

\begin{lemma}\label{lem:short}
Assume that $\cA$ is artinian and $(\cA,\cB)$ satisfies the 
lifting property. Then every object $X \in \cA$ lies
in an exact sequence
\[ 0 \longrightarrow Z \longrightarrow X 
\longrightarrow Y \longrightarrow 0 \]
in $\cA$, where $Z \in \cB$ and $Y \in {^{\perp}\cB}$.
\end{lemma}
\begin{proof}
Consider the exact sequence (\ref{eqn:ext}) and 
use the lifting property to choose a $Z \subset X$
such that $Z \in \cB$ and the composition $Z \to X \to X_{\cB}$
is an epimorphism in $\cA$. Then $X = Z + X^{\cB}$ and hence
$X/Z \cong X^{\cB}/X^{\cB} \cap Z \in {^{\perp}\cB}$.
\end{proof}

Next, we show that every exact sequence in $\cA/\cB$ can
be lifted to an exact sequence in $\cA$, thereby generalizing
\cite[Prop.~3.5]{Brion-II} with a simpler proof:

\begin{lemma}\label{lem:long}
Assume that $\cA$ is artinian and $(\cA,\cB)$ satisfies the 
lifting property. Consider a complex
\begin{equation}\label{eqn:long}
0 \longrightarrow X_n 
\stackrel{\varphi_{n-1}}{\longrightarrow} X_{n-1} 
\stackrel{\varphi_{n-2}}{\longrightarrow} \cdots
\stackrel{\varphi_1}{\longrightarrow} X_1 \longrightarrow 0
\end{equation}
in $\cA/\cB$. 

\begin{enumerate}

\item[{\rm (i)}] There exists a complex
\begin{equation}\label{eqn:longquot}
0 \longrightarrow Y_n 
\stackrel{f_{n-1}}{\longrightarrow} Y_{n-1} 
\stackrel{f_{n-2}}{\longrightarrow} \cdots
\stackrel{f_1}{\longrightarrow} Y_1 \longrightarrow 0
\end{equation}
in $\cA$, together with epimorphisms $g_i : X_i \to Y_i$ in
$\cA$ for $i = 1, \ldots, n$, such that $Y_i \in {^\perp\cB}$
and $\Ker(g_i) \in \cB$ for $i = 1, \ldots, n$, and the diagram
\begin{equation}\label{eqn:longcom}
\xymatrix{
0 \ar[r] & X_n \ar[r]^{\varphi_{n-1}} \ar[d]^{Q(g_n)}
& X_{n-1} \ar[r] \ar[d]^{Q(g_{n-1})}& \cdots \ar[r] & 
X_2 \ar[r]^{\varphi_1} \ar[d]^{Q(g_2)}  
& X_1 \ar[r] \ar[d]^{Q(g_1)} & 0 \\
0 \ar[r] & Y_n \ar[r]_{Q(f_{n-1})} & Y_{n-1} \ar[r] 
& \cdots \ar[r]  & Y_2 \ar[r]_{Q(f_1)}  
& Y_1 \ar[r] & 0
}
\end{equation}
commutes in $\cA/\cB$.

\item[{\rm (ii)}] If the complex (\ref{eqn:long}) is exact in
$\cA/\cB$, then (\ref{eqn:longquot}) may be chosen to be exact 
in $\cA$.

\end{enumerate}

\end{lemma}

\begin{proof}
(i) If $n = 1$, then we just have an object $X \in \cA$
and the assertion follows from Lemma \ref{lem:short}.

For an arbitrary $n$, this lemma yields subobjects 
$X'_i \subset X_i$ such that $X'_i \in \cB$ and 
$X_i/X'_i  \in {^{\perp}\cB}$ for $i = 1, \ldots, n$. Replacing
each $X_i$ with $X_i/X'_i$, we may thus assume that 
$X_i \in {^{\perp}\cB}$. In view of Lemma \ref{lem:equiv}, 
$\varphi_{n-1}$ is represented by an
$f_{n-1} \in \Hom_{\cA}(X_n,X_{n-1}/X'_{n-1})$ for some
$X'_{n-1} \subset X_{n-1}$ such that $X'_{n-1} \in \cB$.
Then $X_{n-1}/X'_{n-1} \in {^{\perp}\cB}$ and the quotient
morphism $X_{n-1} \to X_{n-1}/X'_{n-1}$ is an isomorphism
in $\cA/\cB$. Therefore, we may replace $X_{n-1}$ with
$X_{n-1}/X'_{n-1}$, and hence assume that 
$\varphi_{n-1} = Q(f_{n-1})$ for some 
$f_{n-1} \in \Hom_{\cA}(X_n,X_{n-1})$. Arguing similarly with
$\varphi_{n-2}, \ldots, \varphi_1$, we reduce to the case 
where there exist $f_i \in \Hom_{\cA}(X_i,X_{i-1})$ such
that $\varphi_i = Q(f_i)$ for $i = 1, \ldots, n-1$. As 
$\varphi_i \circ \varphi_{i+1} = 0$ for $i = 1, \ldots, n-2$,
we obtain $f_i \circ f_{i+1} = 0$ by Lemma \ref{lem:equiv} again.

(ii) By (i), we may assume that $X_i \in {^{\perp}\cB}$ and
$\varphi_i = Q(f_i)$ for $i = 1, \ldots, n-1$, where 
\begin{equation}\label{eqn:com} 
0 \longrightarrow X_n 
\stackrel{f_{n-1}}{\longrightarrow} X_{n-1} \cdots 
\stackrel{f_1}{\longrightarrow} X_1 \longrightarrow 0 
\end{equation}
is a complex in $\cA$ with homology objects in $\cB$.
Then $f_1$ is an epimorphism in $\cA$ by Lemma \ref{lem:equiv}
again. Also, since $\Ker(f_1)/\Im(f_2) \in \cB$, we may choose
an $X'_2 \subset \Ker(f_1) \subset X_2$ such that $X'_2 \in \cB$ 
and $\Ker(f_1) = X'_2 + \Im(f_2)$. This yields a commutative
diagram in $\cA$
\[ \xymatrix{
X_3 \ar[r]^{f_2} \ar[d]^{\id} & X_2 \ar[r]^{f_1} \ar[d]^{h} 
& X_1 \ar[r] \ar[d]^{\id} & 0 \\
X_3 \ar[r]^-{g_2} & X_2/X'_2 \ar[r]^-{g_1} & X_1 \ar[r] & 0,
} \]
where $h$ denotes the quotient. Since 
\[ \Ker(g_1) = h(\Ker(f_1)) = h(\Im(f_2)) = \Im(g_2), \]
the bottom sequence is exact in $\cA$. Thus, we may replace 
(\ref{eqn:com}) with the complex
\[ 0 \longrightarrow X_n \longrightarrow \cdots \longrightarrow
X_3 \stackrel{g_2}{\longrightarrow} X_2/X'_2 
\stackrel{g_1}{\longrightarrow} X_1 \longrightarrow 0, \]
and hence assume that (\ref{eqn:com}) is exact at $X_2$ and 
$X_1$. We now argue as above at $X_3$; this yields a new
complex which is exact at $X_3$, and where the tail
$\Im(f_2) \to X_2 \to X_1 \to 0$ is unchanged. Iterating this
construction completes the proof.
\end{proof}

\subsection{Pro-artinian categories}
\label{subsec:proart}

We continue to consider an artinian abelian category $\cA$. Recall 
the definition of the \emph{pro category} $\Pro(\cA)$: its objects 
are the filtered projective systems of objects of $\cA$, and we have
\begin{equation}\label{eqn:hompro} 
\Hom_{\Pro(\cA)}( \{ X_i \}, \{Y_j \} ) = 
\lim_{\leftarrow j} \lim_{\to, i} \big( \Hom_{\cA}(X_i,Y_j) \big). 
\end{equation}
The category $\cA$ is equivalent to a Serre subcategory of 
$\Pro(\cA)$; the latter is abelian, has enough projectives, 
and every artinian object of $\Pro(\cA)$ is isomorphic to an object 
of $\cA$. Moreover, $\Pro(\cA)$ has exact projective limits
(in particular, it has arbitrary products), and every object of 
$\Pro(\cA)$ is the filtered inverse limit of its artinian quotients
(see \cite[\S 1]{Oort64} for these results). 

Thus, $\Pro(\cA)$ is a \emph{pro-artinian category} as defined in
\cite[V.2.2]{DG}; this is the dual notion to that of a locally noetherian
category introduced in \cite[II.4]{Gabriel}. Conversely, every 
pro-artinian category $\cC$ is equivalent to $\Pro(\cA)$, where 
$\cA$ denotes the full subcategory of $\cC$ consisting of 
the artinian objects (see \cite[V.2.3.1]{DG} or 
\cite[II.4.Thm.~1]{Gabriel}). 

Also, recall the natural isomorphisms 
\begin{equation}\label{eqn:pro} 
\Ext^i_{\cA}(X,Y) \stackrel{\cong}{\longrightarrow}
\Ext^i_{\Pro(\cA)}(X,Y) 
\end{equation}
for all $X,Y \in \cA$ and all $i \geq 0$ (see \cite[Thm.~3.5]{Oort64}).
Here the higher extension groups $\Ext^i_{\cA}$ are defined via 
equivalence classes of Yoneda extensions in the abelian category
$\cA$.

For later use, we record a characterization of projective objects of
$\Pro(\cA)$:

\begin{lemma}\label{lem:pro}
The following are equivalent for an object $P \in \Pro(\cA)$:

\begin{enumerate}

\item[{\rm (i)}] $P$ is projective in $\Pro(\cA)$.

\item[{\rm (ii)}] $\Ext^1_{\Pro(\cA)}(P,Z) = 0$ for any $Z \in \cA$.

\item[{\rm (iii)}] For any epimorphism $X \to Y$ in $\cA$, the
induced map 
\[ \Hom_{\Pro(\cA)}(P,X) \longrightarrow \Hom_{\Pro(\cA)}(P,Y) \]
is surjective.

\end{enumerate}

\end{lemma}

\begin{proof}
(i) $\Rightarrow$ (ii) This is well known.

(ii) $\Rightarrow$ (iii) This follows from the long exact sequence
for higher extension groups.

(iii) $\Rightarrow$ (i) This is a consequence of \cite[V.2.3.5]{DG}.
\end{proof}

Next, let $\cB$ be a Serre subcategory of the artinian abelian 
category $\cA$. By (\ref{eqn:hompro}), we may identify $\Pro(\cB)$ 
with the full subcategory of $\Pro(\cA)$ consisting of the filtered
projective systems of objects of $\cB$. 

\begin{lemma}\label{lem:prob}
With the above notation and assumptions, the following conditions
are equivalent for $X \in \Pro(\cA)$:

\begin{enumerate}

\item[{\rm (i)}] $X \in \Pro(\cB)$.

\item[{\rm (ii)}] Every artinian quotient of $X$ is an object of $\cB$.

\end{enumerate}

Moreover, $\Pro(\cB)$ is a Serre subcategory of $\Pro(\cA)$, stable under
inverse limits.
\end{lemma}

\begin{proof}
We find it easier to check the dual statement, where $\cA$ is a
noetherian abelian category, $\cB$ a Serre subcategory, and
$\Pro(\cA)$ is replaced with the ind category $\Ind(\cA)$
(a locally noetherian category).

Let $X \in \Ind(\cB)$. Then $X$ is a direct limit of objects of 
$\cB$, and hence a sum of subobjects in $\cB$ as the latter is
a Serre subcategory. Thus, every noetherian subobject $Y$ of $X$
is a finite sum of subobjects in $\cB$, and hence $Y \in \cB$. 
Conversely, if every noetherian subobject of $X$ is in $\cB$,
then $X \in \Ind(\cB)$ as $X$ is the sum of its noetherian 
subobjects. This shows the equivalence (i)$\Leftrightarrow$(ii).

By this equivalence, $\Ind(\cB)$ is stable under taking subobjects. 
We show that it is stable under taking quotients as well: 
let $f : X \to Y$ be an epimorphism in $\Ind(\cA)$, where 
$X \in \Ind(\cB)$. Consider a noetherian subobject $Y'$ of $Y$; 
then the pull-back $f^{-1}(Y')$ is an object of $\Ind(\cB)$.
Thus, $Y'$ is a sum of subobjects in $\cB$ (since so is
$f^{-1}(Y')$). So $Y'$ is a finite such sum, and $Y' \in \cB$. 

Next, consider an exact sequence $0 \to X_1 \to X \to X_2 \to 0$
in $\Ind(\cA)$, where $X_1, X_2 \in \Ind(\cB)$. Let $X'$ be a
noetherian subobject of $X$, then we have an exact sequence
$0 \to X'_1 \to X' \to X'_2 \to 0$, where $X'_i$ is a noetherian
subobject of $X_i $ for $i = 1,2$. Thus, $X'_i \in \cB$, and 
$X' \in \cB$ as well. So $\Ind(\cB)$ is stable under extensions.

Finally, to show that $\Ind(\cB)$ is stable under direct limits,
it suffices to check the stability under direct sums. Let 
$(X_i)_{i \in I}$ be a family of objects in $\Ind(\cB)$, and
$Y$ a noetherian subobject of $\oplus_{i \in I} \, X_i$. Note 
that $Y$ is the union of its subobjects 
$Y_J := Y \cap \oplus_{j \in J} \, X_j$, where $J$ runs over
the finite subsets of $I$. Choosing $J$ such that $Y_J$ is 
maximal, we easily obtain that $Y = Y_J$; then $Y$ is 
a subobject of $\oplus_{j \in J} \, X_j$. Hence $Y \in \cB$, 
since $\Ind(\cB)$ is stable under finite direct sums.
\end{proof}

Given an artinian abelian category $\cA$ and a Serre subcategory 
$\cB$, the quotient functor $Q : \cA \to \cA/\cB$ extends uniquely 
to an exact functor
\[ \Pro(Q) :  \Pro(\cA) \longrightarrow \Pro(\cA/\cB)\]
which commutes with filtered inverse limits, in view of the dual 
statements to \cite[Prop.~6.1.9, Cor.~8.6.8]{KS}.
Since $\Pro(Q)$ sends every object of $\Pro(\cB)$ to
zero, it factors uniquely through an exact functor
\[ R : \Pro(\cA)/\Pro(\cB) \longrightarrow \Pro(\cA/\cB). \]

\begin{proposition}\label{prop:proab}
With the above notation and assumptions, $R$ is an equivalence 
of categories.
\end{proposition}

\begin{proof}
By the dual statement of \cite[III.4.Prop.~8]{Gabriel},
the subcategory $\Pro(\cB)$ of $\Pro(\cA)$ is 
\emph{colocalizing}, i.e., $\Pro(Q)$ has a left adjoint.
In view of the dual statement of \cite[III.4.Cor.~1]{Gabriel},
it follows that the quotient category $\Pro(\cA)/\Pro(\cB)$ 
is pro-artinian. By the structure theorem for these 
categories (the dual statement of \cite[II.4.Thm.~1]{Gabriel}),
it suffices to show that $R$ restricts to an equivalence 
of the full subcategories consisting of artinian objects. Also, 
the isomorphism classes of artinian objects of $\Pro(\cA/\cB)$ 
are exactly those of objects of $\cA/\cB$. The latter may be 
viewed as a full subcategory of $\Pro(\cA)/\Pro(\cB)$ by Lemma
\ref{lem:prob}; moreover, $R$ restricts to the identity functor on 
$\cA/\cB$. Thus, it suffices in turn to show that the isomorphism 
classes of artinian objects of $\Pro(\cA)/\Pro(\cB)$ are exactly 
those of $\cA/\cB$. 

Let $X$ be an object of $\cA/\cB$, or equivalently of $\cA$. 
We show that $X$ is artinian in $\Pro(\cA)/\Pro(\cB)$ by
adapting arguments from Subsection \ref{subsec:quotients}. 
Consider a monomorphism $\varphi : Y \to X$ in $\Pro(\cA)/\Pro(\cB)$.
Then $\varphi$ is represented by a morphism $f: Y' \to X/X'$
in $\Pro(\cA)$, where $Y' \subset Y$, $Y/Y' \in \Pro(\cB)$, 
$X' \subset X$ and $X' \in \Pro(\cB)$; moreover, 
$\Ker(f) \in \Pro(\cB)$. Replacing $Y$ with an isomorphic object
in $\Pro(\cA)/\Pro(\cB)$, we may thus assume that $\varphi$
is represented by $f : Y \to X/X'$ for some $X' \in \Pro(\cB)$.
Then $X' $ is an object of $\cA$, and hence of $\cB$. 
So $\varphi$ is represented by a monomorphism in $\cA/\cB$.
Since the latter category is artinian (Proposition \ref{prop:artin}), 
it follows that $X$ is indeed artinian in $\Pro(\cA)/\Pro(\cB)$.

Conversely, let $X$ be an object of $\Pro(\cA)$ such that
$\Pro(X)$ is artinian in $\Pro(\cA)/\Pro(\cB)$. 
Consider the family $(X_i)_{i \in I}$
of subobjects of $X$ in $\Pro(\cA)$ such that $X/X_i \in \cA$.
Then $\cap_{i \in I} \, X_i = 0$ in $\Pro(\cA)$, hence
$\cap_{i \in I} \, \Pro(Q)(X_i) = 0$ in $\Pro(\cA)/\Pro(\cB)$,
since $\Pro(Q)$ is exact and commutes with inverse limits 
(by the dual statement to \cite[III.4.Prop.~9]{Gabriel}). 
But $(\Pro(Q)(X_i))_{i \in I}$ has a minimal element 
in $\Pro(\cA)/\Pro(\cB)$, say $\Pro(Q)(X_{i_0})$.
Hence $\Pro(Q)(X_{i_0})$ is zero in that quotient category, 
i.e., $X_{i_0} \in \Pro(\cB)$. Thus, $X/X_{i_0}$ is an object
of $\cA$, isomorphic to $X$ in $\Pro(\cA)/\Pro(\cB)$.
\end{proof}

\begin{remark}\label{rem:coloc}
In view of the dual statement to \cite[III.4.Prop.~10]{Gabriel},
every colocalizing subcategory $\cC$ of $\Pro(\cA)$ is 
pro-artinian; moreover, assigning to $\cC$ its full subcategory
of artinian objects defines a bijective correspondence between
colocalizing subcategories of $\Pro(\cA)$ and Serre subcategories
of $\cA$. One may check that the inverse bijection is 
$\cB \mapsto \Pro(\cB)$, by using Lemma \ref{lem:prob} 
and \cite[III.3.Cor.~1]{Gabriel}.
\end{remark}

\begin{lemma}\label{lem:projective}
Let $\cA$ be an artinian abelian category, and $\cB$ a Serre
subcategory. Then the pair $(\cA,\cB)$ satisfies the lifting 
property if and only if every projective object of $\Pro(\cB)$ 
is projective in $\Pro(\cA)$. Under this assumption, 
every projective object of $\Pro(\cA)$ is sent by $\Pro(Q)$ 
to a projective object of $\Pro(\cA/\cB)$. 
\end{lemma}

\begin{proof}
Assume that the pair $(\cA,\cB)$ satisfies the lifting property.
Let $P$ be a projective object of $\Pro(\cB)$. 
By Lemma \ref{lem:pro}, to show that $P$ is projective 
in $\Pro(\cA)$, it suffices to check that for any epimorphism 
$f: X \to Y$ in $\cA$ and any morphism $h : P \to Y$ 
in $\Pro(\cA)$, there exists $g : P \to X$ in $\Pro(\cA)$ 
such that $g \circ f = h$. Write $P$ as a projective system
$(P_i,f_{ij})$, where $P_i \in \cB$ and 
$f_{ij} \in \Hom_{\cB}(P_j,P_i)$ for $i \leq j$; then
$h \in \lim_{\to, i} (\Hom_{\cA}(P_i,Y))$. Choose a representative
$h_i : P_i \to Y$ and let $Y' := \Im(h_i)$. Then $Y'$ is a
subobject of $Y$ and $Y' \in \cB$;
hence there exists a subobject $X'$ of $f^{-1}(Y')$ such that 
$X' \in \cB$ and $f$ pulls back to an epimorphism
$f' : X' \to Y'$ in $\cB$. Since $P$ is projective in $\Pro(\cB)$, 
there exist $j \geq i$ and $g_j : P_j \to X'$ such that 
$g_j \circ f' = h_i \circ f_{ij}$. This yields the desired lift 
$g : P \to X$.

Conversely, assume that every projective object in $\Pro(\cB)$
is projective in $\Pro(\cA)$. Let $f: X \to Y$ be an epimorphism
in $\cA$, where $Y \in \cB$. Since $\Pro(\cB)$ has enough
projectives, there exists an epimorphism $g: P \to Y$ in that category, 
where $P$ is projective. By assumption, $g$ lifts to a morphism 
$h : P \to X$. Denote by $X'$ the image of $h$; then $X'$ is 
an artinian object of $\Pro(\cB)$, and hence an object of $\cB$ 
by Lemma \ref{lem:prob}. Moreover, the composition 
$X' \to X \to Y$ is an epimorphism. Thus, the pair $(\cA,\cB)$ 
satisfies the lifting property. 

The remaining assertion is a consequence of the dual statement 
to \cite[III.3.Cor.~3]{Gabriel}. Alternatively, it can be proved by 
adapting the above argument: let $P = (P_i,f_{ij})$ be a projective 
object of $\Pro(\cA)$. Consider an epimorphism $\varphi : X \to Y$ 
in $\cA/\cB$ and a morphism $\psi_i : P_i \to Y$ in $\cA/\cB$. 
Replacing $X$ (resp.~$Y$) with $X^{\cB}$ (resp.~$Y^{\cB}$), 
we may assume that $\varphi$ is represented by an epimorphism 
$f: X \to Y/Y'$ in $\cA$, for some $Y' \subset Y$ such that 
$Y' \in \cB$ (Lemma \ref{lem:equiv}). Also, in view of Lemma 
\ref{lem:homq}, $\psi_i$ is represented by a morphism 
$h_i : P_i \to Y/Y''$ in $\cA$, for some $Y'' \subset Y$ such that 
$Y'' \in \cB$. By considering $Y/(Y'+Y'')$, we may assume that 
$Y'' = Y'$. Since $P$ is projective in $\Pro(\cA)$, there exist 
$j \geq i$ and a morphism $h_j : P_j \to X$ in $\cA$ that lifts 
$h_i : P_i \to Y/Y'$. Thus, $Q(h_j) \in \Hom_{\cA/\cB}(P_j,X)$ lifts 
$\psi_i \in \Hom_{\cA/\cB}(P_i,Y)$. 
\end{proof}

\subsection{Homological dimension}
\label{subsec:hd}

Recall that the homological dimension of an abelian category
$\cA$ is the smallest non-negative integer $n$ (if it exists) 
such that 
$\Ext^i_{\cA}(X,Y) = 0$ for all $X,Y \in \cA$ and $i > n$. 
We then set $n =: \hd(\cA)$. If no such integer $n$ exists, then 
we set $\hd(\cA) := \infty$.

When $\cA$ has enough projectives, its homological dimension
is the supremum of the lengths of minimal projective resolutions 
of all objects (see e.g.~\cite[Lem.~4.1.6]{Weibel}). This holds
in particular when $\cA$ is the pro category of an artinian
abelian category (see \cite[p.~229]{Oort66}.

Given a family of abelian categories $(\cA_i)_{i \in I}$,
we clearly have
\begin{equation}\label{eqn:hdsum}
\hd \big( \textstyle{\bigoplus_{i \in I}} \; \cA_i \big) 
= \sup_{i \in I} \big( \hd(\cA_i) \big).
\end{equation}

We will also need the following observation:

\begin{lemma}\label{lem:hdproj}
Let $\cA$ be an artinian abelian category. Then 
\[ \hd(\cA) = \hd(\Pro(\cA)). \]
\end{lemma}

\begin{proof}
In view of the isomorphism (\ref{eqn:pro}), we have
$\hd(\cA) \leq \hd(\Pro(\cA))$.

To show the opposite inequality, we may assume that $\hd(\cA)$
is finite, say $n$. Let $X \in \Pro(\cA)$, then we may write $X$ 
as a filtered projective system $(X_i,\varphi_{ij})$, where 
$X_i \in \cA$. By \cite[V.2.3.9]{DG}, the natural map
\[ \lim_{\to} \big( \Ext^m_{\Pro(\cA)}(X_i,Y) \big) \longrightarrow 
\Ext^m_{\Pro(\cA)}(X,Y) \]
is an isomorphism for any $m \geq 0$ and any $Y \in \cA$.
As a consequence, we get that  $\Ext^m_{\Pro(\cA)}(X,Y) = 0$ 
for any $m > n$ and $Y \in \cA$. 

If $n = 0$, then it follows that $X$ is projective,
in view of Lemma \ref{lem:pro}. Thus, $\hd(\Pro(\cA)) = 0$ as desired.

If $n \geq 1$, then we consider an exact sequence in $\Pro(\cA)$
\[ 0 \longrightarrow X_1 \longrightarrow P_0 
\longrightarrow X \longrightarrow 0, \]  
where $P_0$ is projective. Then the natural map
\[ \Ext^m_{\Pro(\cA)}(X_1,Y) \longrightarrow 
\Ext^{m+1}_{\Pro(\cA)}(X,Y) \]
is an isomorphism for any $m > 0$ and $Y \in \Pro(\cA)$. Thus
$\Ext^m_{\Pro(\cA)}(X_1,Y) = 0$ for any $m > \max(n-1,0)$
and $Y \in \cA$. By induction on $n$, it follows that $X$ admits
a projective resolution in $\Pro(\cA)$
\[ 0 \longrightarrow P_n \longrightarrow \cdots 
\longrightarrow P_0 \longrightarrow X \longrightarrow 0. \]
Thus, $\hd(\Pro(\cA)) \leq n$.
\end{proof}

Our next statement provides a partial answer to a long-standing
question of Oort (see \cite[II.15.2]{Oort66}):

\begin{proposition}\label{prop:hdmax}
Let $\cA$ be an artinian abelian category, and $\cB$ a Serre 
subcategory. If $(\cA,\cB)$ satisfies the lifting property, then
\begin{equation}\label{eqn:geq} 
\hd(\cA) \geq \max(\hd(\cB),\hd(\cA/\cB)).
\end{equation}
\end{proposition}

\begin{proof}
By Lemma \ref{lem:projective} and the isomorphism (\ref{eqn:pro}), 
the natural map
\[ \Ext^i_{\cB}(X,Y) \longrightarrow \Ext^i_{\cA}(X,Y) \]
is an isomorphism for all $X,Y \in \cB$ and $i \geq 0$. Thus,
$\hd(\cB) \leq \hd(\cA)$. 

The inequality $\hd(\cA/\cB) \leq \hd(\cA)$ is obtained by combining
Lemmas \ref{lem:projective} and \ref{lem:hdproj}, and using the
characterization of the homological dimension by lengths of projective
resolutions in the pro category. That inequality can also be deduced 
from Lemma \ref{lem:long} as follows: 
we may again assume that $\hd(\cA) =: n$ is finite. Let $X,Y \in \cA$
and $\xi \in \Ext^{n+1}_{\cA/\cB}(X,Y)$. Then $\xi$ is represented
by an exact sequence in $\cA/\cB$
\begin{equation}\label{eqn:longex}
0 \longrightarrow Y \longrightarrow X_{n+1} \longrightarrow
\cdots \longrightarrow X_1 \longrightarrow X \longrightarrow 0. 
\end{equation} 
In view of Lemma \ref{lem:long} (ii), we may assume (possibly by
replacing $X,Y$ with isomorphic objects in $\cA/\cB$) that
(\ref{eqn:longex}) is represented by an exact sequence in $\cA$.
Since $\Ext^{n+1}_{\cA}(X,Y) = 0$ and the quotient functor 
$Q : \cA \to \cA/\cB$ is exact, it follows that $\xi = 0$.
\end{proof}

We will show in Theorem \ref{thm:hd} that the equality holds in 
(\ref{eqn:geq}) when $\cB$ is the $S$-primary torsion subcategory 
of $\cA$ for some set $S$ of prime numbers. But the inequality 
(\ref{eqn:geq}) is generally strict, as shown by the following:

\begin{example}\label{ex:A2}
We freely use some notions and results from the representation theory 
of quivers, which can be found e.g. in \cite[4.1]{Benson}.
Consider the quiver $Q$:  $\xymatrix{ {1} \ar[r] & {2} }$ and let $\cA$
be the category of finite-dimensional representations of $Q$ over 
a field $k$. Then $\cA$ is an artinian and noetherian abelian category, 
and $\hd(\cA) = 1$. The simple objects of $\cA$ are isomorphic to 
$S_1$: $\xymatrix{ {k} \ar[r] & {0} }$ or
$S_2$: $\xymatrix{ {0} \ar[r] & {k} }$; moreover, $S_2$ is
projective and $S_1$ is not. Let $\cB$ be the full subcategory of
$\cA$ with objects the sums of copies of $S_2$. Then $\cB$ is a 
Serre subcategory, equivalent to the category $k$-mod 
of finite-dimensional $k$-vector spaces; thus, $\hd(\cB) = 0$.
Since $\cB$ consists of projective objects of $\cA$, it 
satisfies the lifting property. Moreover, the quotient category 
$\cA/\cB$ is again equivalent to $k$-mod, and hence 
$\hd(\cA/\cB) = 0$.
\end{example}

\subsection{$S$-primary torsion subcategories}
\label{subsec:torsion}

Let $\bP$ be the set of all prime numbers, and $S$
a subset of $\bP$. We denote by $\langle S \rangle$ 
the set of positive integers $n$ such that all prime factors 
of $n$ are in $S$. We say that an object $X$ of an abelian
category is $S$-\emph{torsion} (resp.~$S$-\emph{divisible}) 
if the multiplication
\[ n_X : X \longrightarrow X, \quad x \longmapsto n x \]
is zero for some $n \in \langle S \rangle$ (resp.~is an 
epimorphism for all $n \in \langle S \rangle$). When $S$ is the
whole $\bP$, we just say that $X$ is torsion, resp.~divisible.

Next, let $\cA$ be an artinian abelian category. Denote by $\cA_S$
(resp.~$\cA^S$) the full subcategory of $\cA$ formed by the 
$S$-primary torsion (resp.~$S$-divisible) objects. One may readily 
check that $\cA_S$ is a Serre subcategory of $\cA$; moreover,
$\cA^S = {^{\perp}\cA_S}$ with the notation of Subsection 
\ref{subsec:quotients}. In view of Lemma \ref{lem:exact}, 
it follows that every $X \in \cA$ lies in 
a unique exact sequence
\begin{equation}\label{eqn:torsion}
0 \longrightarrow X^S \longrightarrow X \longrightarrow X_S
\longrightarrow 0, 
\end{equation}
where $X^S$ is $S$-divisible and $X_S$ is $S$-primary torsion. 
With this convention, the torsion subcategory $\cA_{\tors}$ 
considered in the introduction is denoted by $\cA_{\bP}$.

Also, note the equivalence of categories
\begin{equation}\label{eqn:sum}
\textstyle{\bigoplus_{p \in S}} \; \cA_p 
\stackrel{\cong}{\longrightarrow} \cA_S,
\end{equation}
which yields equivalences of categories 
$\cA_{S_1} \times \cA_{S_2} \cong \cA_{S_1 \cup S_2}$
for any two disjoint subsets $S_1,S_2$ of $\bP$.

For later use, we record the following observation:

\begin{lemma}\label{lem:quot}

\begin{enumerate}

\item[{\rm (i)}]
Let $\cB$ be a Serre subcategory of $\cA_S$. Then 
$\cA_S/\cB$ is equivalent to $(\cA/\cB)_S$. 

\item[{\rm (ii)}] Let $T$ be a subset of $S$. Then 
$\cA_S/\cA_T \cong \cA_{S \setminus T}$.

\end{enumerate}

\end{lemma}

\begin{proof}
(i) We may view $\cA_S/\cB$ and $(\cA/\cB)_S$ as full subcategories
of $\cA/\cB$; thus, it suffices to show that they have the same
objects. Let $X \in (\cA/\cB)_S$, then there exists 
$n \in \langle S \rangle$ such that $n_X = 0$ in $\cA/\cB$. 
By \cite[III.1.Lem.~2]{Gabriel}, this is equivalent to the condition 
that $\Im_{\cA}(n_X) \in \cB$. Then $\Im_{\cA}(n_X)$ is 
$S$-primary torsion, and hence $X$ is $S$-torsion as well, 
i.e., $X \in \cA_S/\cB$. The converse implication is obvious.

(ii) This follows readily from the decomposition (\ref{eqn:sum}). 
\end{proof}

\begin{lemma}\label{lem:Slift}
With the above notation, the pair $(\cA,\cA_S)$ satisfies the lifting property.
\end{lemma}

\begin{proof}
Consider an exact sequence in $\cA$
\[ 0 \longrightarrow Z \longrightarrow X \longrightarrow Y
\longrightarrow 0, \]
where $Y \in \cA_S$. Then $Y$ is a quotient of $X_S$, and hence
it suffices to check the lifting property for the exact
sequence (\ref{eqn:torsion}).

We may choose $n \in \langle S \rangle$ such that $n_{X_S} = 0$;
then $n_{X^S}$ is an epimorphism as $X^S$ is $n$-divisible. 
Applying the snake lemma to the commutative diagram
\[ \xymatrix{
0 \ar[r] & X^S \ar[r] \ar[d]^{n_{X^S}} & X \ar[r] \ar[d]^{n_X} 
& X_S \ar[r] \ar[d]^{n_{X_S}} & 0 \\
0 \ar[r] & X^S \ar[r] & X \ar[r] & X_S \ar[r] & 0,
}
\]
we obtain an epimorphism $\Ker(n_X) \to X_S$. Since 
$\Ker(n_X) \in \cA_S$, this completes the proof. 
\end{proof}

The following result extends \cite[Prop.~3.6 (ii)]{Brion-II}
to our categorical setting:

\begin{lemma}\label{lem:hom}
Let $X,Y \in \cA$. Then the natural map
\[ Q : \Hom_{\cA}(X,Y) \longrightarrow \Hom_{\cA/\cA_S}(X,Y) \]
induces an isomorphism
\[ S^{-1} \Hom_{\cA}(X,Y) \stackrel{\cong}{\longrightarrow} 
\Hom_{\cA/\cA_S}(X,Y). \]
\end{lemma}

\begin{proof}
For any $n \in \langle S \rangle$, the multiplication map  
$n_X$ yields an isomorphism in $\cA/\cA_S$, since $\Ker_{\cA}(n_X)$ 
and $\Coker_{\cA}(n_X)$ are $S$-primary torsion. Thus,
$\End_{\cA/\cA_S}(X)$ is an algebra over the localization
$S^{-1} \bZ$, and $\Hom_{\cA/\cA_S}(X,Y)$ is a module over 
$S^{-1} \bZ$. This yields a natural map
\[ \gamma : S^{-1} \Hom_{\cA}(X,Y) \longrightarrow
\Hom_{\cA/\cA_S}(X,Y). \]

To show that $\gamma$ is an isomorphism, we may replace
$X$ with $X^S$. Indeed, the map 
$\Hom_{\cA/\cA_S}(X,Y) \to \Hom_{\cA/\cA_S}(X^S,Y)$
is an isomorphism, since the inclusion $X^S \to X$ is an isomorphism
in $\cA/\cA_S$; also, the map 
$S^{-1}\Hom_{\cA}(X,Y) \to S^{-1}\Hom_{\cA}(X^S,Y)$
is an isomorphism as well, since $\Ext^i_{\cA}(X_S,Y)$ is 
$S$-primary torsion for all $i \geq 0$. Likewise, we may replace 
$Y$ with $Y^S$. 

Since now $X$ is $S$-divisible, we have an isomorphism 
\[ X/\Ker_{\cA}(n_X) \stackrel{\cong}{\longrightarrow} X \]
in $\cA$ for any $n \in \langle S \rangle$. Denote its inverse by 
\[ n_X^{-1} : X \stackrel{\cong}{\longrightarrow} X/\Ker_{\cA}(n_X). \]
Then $\gamma(n^{-1}f) = Q(n_Y^{-1} \circ f)$
for any such $n$ and $f \in \Hom_{\cA}(X,Y)$. 

We now show that $\gamma$ is injective. Given $n$ and $f$ 
as above such that $\gamma(n^{-1} f) = 0$, we have
$Q(n_Y^{-1} \circ f) = 0$ in $\Hom_{\cA/\cA_S}(X,Y)$, and hence
$n_Y^{-1} \circ f = 0$ in $\Hom_{\cA}(X,Y)$ in view of Lemma
\ref{lem:equiv}. As $n_Y$ is an epimorphism, it follows that
$f = 0$.

Finally, we show that $\gamma$ is surjective. Let 
$f \in \Hom_{\cA/\cA_S}(X,Y)$. By Lemma \ref{lem:equiv} again,
$f$ is represented by $\varphi \in \Hom_{\cA}(X,Y/Y')$
for some $S$-primary torsion subobject $Y' \subset Y$. Choose 
$n \in \langle S \rangle$ such that $n_{Y'} = 0$; then
$Y' \subset \Ker_{\cA}(n_Y)$ and hence $f$ is also represented
by some $\psi \in \Hom_{\cA}(X,Y/\Ker_{\cA}(n_Y))$. It follows
that $f = \gamma(n^{-1}\psi)$.
\end{proof}

\begin{lemma}\label{lem:ext}
We have natural isomorphisms
\[ S^{-1} \Ext^i_{\cA}(X,Y) \stackrel{\cong}{\longrightarrow}
\Ext^i_{\cA/\cA_S}(X,Y) \]
for all $i \geq 0$ and all $X,Y \in \cA$. 
\end{lemma}

\begin{proof}
We may choose a projective resolution of $X$ in $\Pro(\cA)$
\[ \cdots \longrightarrow P_1 \longrightarrow P_0 
\longrightarrow X \longrightarrow 0. \]
Then $\Ext^i_{\cA}(X,Y)$ is the $i$th homology group of the complex
\[ 0 \longrightarrow \Hom_{\Pro(\cA)}(P_0,Y) 
\longrightarrow \Hom_{\Pro(\cA)}(P_1,Y) 
\longrightarrow \cdots \]
in view of the isomorphism (\ref{eqn:pro}).
Since $\Pro(Q)(P_0), \Pro(Q)(P_1), \ldots$ 
are projective in $\Pro(\cA/\cA_S)$ (Lemmas \ref{lem:projective} 
and \ref{lem:Slift}) and $\Pro(Q)$ is exact, we see that 
$\Ext^i_{\cA/\cA_S}(X,Y)$ is the $i$th homology group of the complex
\[ 0 \to \Hom_{\Pro(\cA/\cA_S)}(\Pro(Q)(P_0), Y) \to
\Hom_{\Pro(\cA/\cA_S)}(\Pro(Q)(P_1), Y) \to \cdots \]
(we identify $\Pro(Q)(Y) = Q(Y)$ with $Y$). As localizing by $S$ 
is exact, it suffices to construct a natural isomorphism 
\[ S^{-1} \Hom_{\Pro(\cA)}(Z, Y) \stackrel{\cong}{\longrightarrow}
\Hom_{\Pro(\cA/\cA_S)}(\Pro(Q)(Z), Y) \]
for any $Z \in \Pro(\cA)$ and $Y \in \cA$. Write $Z$ as 
the projective system $(Z_i,\varphi_{ij})$, then 
\[ \Hom_{\Pro(\cA)}(Z,Y) = \lim_{\to} \big( \Hom_{\cA}(Z_i,Y) \big) \]
and hence 
\[ S^{-1} \Hom_{\Pro(\cA)}(Z,Y) = 
S^{-1} \lim_{\to} \big( \Hom_{\cA}(Z_i,Y) \big) =
\lim_{\to} \big( S^{-1} \Hom_{\cA}(Z_i,Y) \big). \]
By Lemma \ref{lem:hom}, this is naturally isomorphic to
\[ \lim_{\to} \big( \Hom_{\cA/\cA_S}(Z_i,Y) \big)
= \Hom_{\Pro(\cA/\cA_S)}(\Pro(Q)(Z),Y). \]
\end{proof}

We now come to the main result of this section:

\begin{theorem}\label{thm:hd}
Let $\cA$ be an artinian abelian category, $S$ a set of prime numbers,
and $\cA_S$ the $S$-primary torsion subcategory of $\cA$. Then 
\[ \hd(\cA) = \max(\hd(\cA_S),\hd(\cA/\cA_S)). \]
\end{theorem}

\begin{proof}
Combining Proposition \ref{prop:hdmax} and Lemma \ref{lem:Slift} 
yields the inequality 
$\hd(\cA) \geq \max(\hd(\cA_S),\hd(\cA/\cA_S))$.

To show the opposite inequality, we may of course assume that
$\hd(\cA_S) = d$ and $\hd(\cA/\cA_S) = e$ are both finite.
Let $X,Y \in \cA$. In view of the exact sequence (\ref{eqn:torsion})
for $X$ and $Y$, it suffices to show that $\Ext^i_{\cA}(X',Y') = 0$
for all $i > \max(d,e)$ and all $X' \in \{ X_S,X^S \}$, 
$Y' \in \{ Y_S, Y^S \}$. This vanishing holds when $X' = X_S$ 
and $Y' = Y_S$, since we then have
$\Ext^i_{\cA}(X',Y') = \Ext^i_{\cA_S}(X',Y')$
by Lemmas \ref{lem:projective} and  \ref{lem:Slift}. 

When $X' = X^S$ and $Y' = Y_S$, we choose $n \in \langle S \rangle$
such that $n_{Y_S} = 0$; then $\Ext^i_{\cA}(Z,Y_S)$ is killed by $n$
for all $i \geq 0$ and $Z \in \cA$. Thus, the exact sequence
\[ 0 \longrightarrow \Ker_{\cA}(n_{X^S}) \longrightarrow X^S
\stackrel{n}{\longrightarrow} X^S \longrightarrow 0 \]
yields an exact sequence
\[ 0 \longrightarrow \Ext^i_{\cA}(X^S,Y_S) 
\longrightarrow \Ext^i_{\cA}(\Ker_{\cA}(n_{X^S}), Y_S). \]
Since $\Ker_{\cA}(n_{X^S}) \in \cA_S$, it follows
that $\Ext^i_{\cA}(X^S,Y_S) = 0$ for all $i > d$.

When $X' = X_S$ and $Y' = Y^S$, we obtain similarly an
exact sequence
\[ \Ext^i_{\cA}(X_S,\Ker_{\cA}(n_{Y^S})) \longrightarrow 
  \Ext^i_{\cA}(X_S, Y^S) \longrightarrow 0 \]
for all $i \geq 0$ and $n \in \langle S \rangle$ such that
$n_{X_S} = 0$. Thus, $\Ext^i_{\cA}(X_S,Y^S) = 0$ for all $i > d$.

Finally, we consider the case where $X' = X^S$ and $Y' = Y^S$. 
For all $i > e$, we have
\[ 0 = \Ext^i_{\cA/\cA_S}(X^S,Y^S) = S^{-1} \Ext^i_{\cA}(X^S,Y^S) \]
in view of Lemma \ref{lem:ext}. Thus, for any 
$\xi \in \Ext^i_{\cA}(X^S,Y^S)$, we may find some
$n \in \langle S \rangle$ (depending on $\xi$) such that 
$n \xi = 0$. Then $\xi$ lies in the image of 
$\Ext^i_{\cA}(X^S, \Ker_{\cA}(n_{Y^S}))$, and the latter group 
vanishes if $i \geq d$.
\end{proof}

\section{Isogeny categories of algebraic groups}
\label{sec:isogeny}

\subsection{$S$-isogeny categories}
\label{subsec:Siso}

We return to the setting of the introduction, and consider
a field $k$ of characteristic $p \geq 0$. We choose an algebraic
closure $\bar{k}$ of $k$, and denote by $k_s$ the separable closure 
of $k$ in $\bar{k}$. The Galois group of $k_s/k$ is denoted by 
$\Gamma$; this is a profinite group. 

By an \emph{algebraic $k$-group},
we mean a commutative group scheme of finite type over $k$.
The algebraic $k$-groups are the objects of an abelian category
$\cC = \cC_k$ with morphisms being the homomorphisms of 
$k$-group schemes. Since every descending chain of 
closed subschemes of a $k$-scheme of finite type is stationary,
the category $\cC$ is artinian.
The finite (resp.~finite \'etale) $k$-group schemes 
form a full subcategory $\cF = \cF_k$ (resp. $\cE = \cE_k$) of $\cC$; 
one may easily check that $\cF$ and $\cE$ are Serre subcategories. 
By Cartier duality, $\cE$ is anti-equivalent to the category 
$\Gamma\ldash\mod$ of finite abelian groups equipped with 
a discrete action of $\Gamma$ (see \cite[II.5.1.7]{DG}). 

As in Subsection \ref{subsec:torsion}, we consider a set
$S$ of prime numbers, and the $S$-primary torsion subcategory
$\cC_S$ (resp.~$\cF_S$, $\cE_S$) of $\cC$ (resp.~$\cF$, $\cE$). 
We say that the quotient category $\cC/\cF_S$ is the 
\emph{$S$-isogeny category}, and $\cC/\cE_S$ is the
\emph{\'etale $S$-isogeny category}. Both are artinian abelian
categories in view of Proposition \ref{prop:artin};
the quotient category $\cC/\cF_{\bP}$ is the (full) isogeny category
studied in \cite{Brion-II} and \cite{Brion-III}.

By Lemma \ref{lem:Slift}, the pair $(\cC,\cC_S)$ has the lifting 
property. Also, $(\cC,\cF_{\bP})$ has the lifting  property as well, 
in view of the main result of \cite{Brion} (see also 
\cite[Thm.~3.2]{Lucchini}). As an easy consequence, we obtain:

\begin{lemma}\label{lem:lifting}
The pair $(\cC,\cF_S)$ has the lifting property.
If $k$ is perfect, then $(\cC, \cE_S)$ has the lifting 
property as well. 
\end{lemma}

\begin{proof}
Consider an epimorphism $G \to H \to 0$ in $\cC$, where 
$H \in \cF_S$. As just recalled, there exists a finite subgroup scheme
$G' \subset G$ such that the composition $G' \to G \to H$ is an
epimorphism. Since $G' \cong G'_S \times G'_{S'}$ and the composition
$G'_{S'} \to G \to H$ is zero, it follows that the composition
$G'_S \to G \to H$ is an epimorphism as well.

When $k$ is perfect, the reduced subscheme 
$G'_{S,\red} \subset G'_S$ is an \'etale subgroup scheme 
(as follows from \cite[II.5.2.3]{DG}), and the composition 
$G'_{S,\red} \to G' \to G \to H$ is surjective on $\bar{k}$-rational 
points. Thus, this composition is an epimorphism whenever 
$H \in \cE_S$.
\end{proof}

(If $k$ is imperfect, then the lifting property fails for the
pair $(\cC,\cE_{\bP})$, see \cite[Rem.~3.3]{Brion}).

We now describe the subcategories $\cC_S$, $\cF_S$ and $\cE_S$
of $\cC$. Note first that $\cE_S$ is anti-equivalent to the 
$S$-primary torsion subcategory $\Gamma\ldash\mod_S$ 
of $\Gamma\ldash\mod$. Also, (\ref{eqn:sum}) yields equivalences 
of categories
\begin{equation}\label{eqn:sums} 
\textstyle{\bigoplus_{\ell \in S}} \; \cC_{\ell} \cong \cC_S, \quad
\textstyle{\bigoplus_{\ell \in S}} \; \cF_{\ell} \cong \cF_S,
\quad  \textstyle{\bigoplus_{\ell \in S}} \; \cE_{\ell} \cong \cE_{\ell}. 
\end{equation}
As a consequence, we have equivalences 
\[ \cC_{\bP} \cong \cC_S \times \cC_{S'}, \quad 
\cF_{\bP} \cong \cF_S \times \cF_{S'}, \quad
\cE_{\bP} \cong \cE_S \times \cE_{S'}, 
\]
where $S' := \bP \setminus S$.

Given $G \in \cC$ and $n \in \bZ$,  recall that the multiplication
map  $n_G$ is \'etale if $n$ is prime to $p$ (see \cite[VIIA.8.4]{SGA3}); 
then $\Ker(n_G) \in \cE$. Also, recall that $n_G = 0$ 
when $G$ is finite of order $n$, where the \emph{order} of $G$ 
is the dimension of the $k$-vector space $\cO(G)$ (see 
\cite[VIIA.8.5]{SGA3}). In particular, every finite group
scheme of order prime to $p$ is \'etale. Together with (\ref{eqn:sums}),
this yields
\begin{equation}\label{eqn:CFS} 
\cC_S \cong \begin{cases} \cF_S = \cE_S & \text{if $p \notin S$,} \\
\cC_p \times \cF_{S \setminus \{p \}} =  
\cC_p \times \cE_{S \setminus \{p \}} & \text{if $p \in S$.}
\end{cases} 
\end{equation}

To describe the $p$-primary torsion category $\cC_p$, we need 
some structure results for affine algebraic groups.
Recall from \cite[IV.3.1]{DG} that every such group $G$ 
lies in a unique exact sequence
\begin{equation}\label{eqn:aff}
0 \longrightarrow M \longrightarrow G \longrightarrow U
\longrightarrow 0,
\end{equation}
where $M$ is of multiplicative type and $U$ is unipotent;
moreover, we have $\Hom_{\cC}(M,U) = 0 = \Hom_{\cC}(U,M)$.
We denote by $\cL$ (resp.~$\cM$, $\cU$) the full subcategory of
$\cC$ with objects the affine algebraic groups (resp.~the groups 
of multiplicative type, the unipotent groups); then $\cL$, $\cM$
and $\cU$ are Serre subcategories of $\cC$.
If $k$ is perfect, then the exact sequence (\ref{eqn:aff}) 
has a unique splitting by \cite[IV.3.1]{DG} again; this yields 
an equivalence of categories
\begin{equation}\label{eqn:MU}
\cM \times \cU \stackrel{\cong}{\longrightarrow} \cL.
\end{equation}

\begin{lemma}\label{lem:ptors}
Assume that $p  > 0$.

\begin{enumerate}

\item[{\rm (i)}] The objects of $\cC_p$ are exactly the 
algebraic groups obtained as extensions (\ref{eqn:aff}),
where $M \in \cM_p$ and $U \in \cU$.

\item[{\rm (ii)}] If $k$ is perfect, then 
$\cC_p \cong \cM_p \times \cU$.

\item[{\rm (iii)}] For an arbitrary field $k$, we have
$\cC_p/\cF_p \cong \cU/(\cF_p \cap \cU)$. 

\end{enumerate}

\end{lemma}

\begin{proof}
Let $G \in \cC_p$. Since abelian varieties are divisible, every 
such variety obtained as a subquotient of $G$ is zero. It follows
that $G$ is affine (e.g., by using \cite[Prop.~2.8]{Brion-II}),
and hence an extension of a unipotent group by a $p$-primary torsion 
group of multiplicative type. Conversely, every such extension is 
$p$-primary torsion, since so is every unipotent group. This proves (i). 

The assertions (ii) and (iii) are direct consequences of (i) in view
of the above structure results for affine algebraic groups.
\end{proof}

Next, we describe the subcategories $\cC^S = {^{\perp}\cC_S}$ 
(the $S$-divisible algebraic groups),
${^{\perp}\cF_S}$ and ${^{\perp}\cE_S}$, with the notation of 
Subsections \ref{subsec:quotients} and \ref{subsec:torsion}; here 
$\cF_S$ and $\cE_S$ are viewed as Serre subcategories of $\cC$. 
For this, recall that every algebraic group $G$ lies in 
a unique exact sequence
\begin{equation}\label{eqn:neutral}
0 \longrightarrow G^0 \longrightarrow G \longrightarrow 
\pi_0(G) \longrightarrow 0, 
\end{equation}
where $G^0$ is connected and $\pi_0(G) \in \cE$ 
(see e.g. \cite[II.5.1.8]{DG}). We may now treat the easy case
where $p \notin S$:

\begin{lemma}\label{lem:pnotin}
If $p \notin S$, then 
$\cC^S = {^{\perp}\cF_S} = {^{\perp}\cE_S}$ consists of 
the algebraic groups $G$ such that $\pi_0(G) \in \cE_{S'}$.
\end{lemma}

\begin{proof}
Note that $\cC^S$ consists of the algebraic groups $G$ 
such that every quotient $S$-primary torsion group of $G$ is trivial.
Since such a quotient group is \'etale, it is a quotient
of $\pi_0(G)$, and thus of $\pi_0(G)_S$. So $G \in \cC^S$
if and only if $\pi_0(G)$ is $S'$-primary torsion.
\end{proof}

By contrast, the case where $p \in S$ is much more involved:

\begin{lemma}\label{lem:pin}
Assume that $p \in S$. 

\begin{enumerate}

\item[{\rm (i)}] The objects of $\cC_S$ (resp.~$\cC^S$) 
are exactly the algebraic groups $G$ obtained as extensions
\[ 0 \longrightarrow G_1 \longrightarrow G \longrightarrow
G_2 \longrightarrow 0, \]
where $G_1$ is an $S$-primary torsion group of multiplicative type 
and $G_2$ is unipotent (resp.~$G_1$ is a semi-abelian variety
and $G_2 \in \cE_{S'}$). Moreover, ${^{\perp}\cE_S}$ consists 
of the algebraic groups $G$ such that $\pi_0(G) \in \cE_{S'}$.

\item[{\rm (ii)}] When $k$ is perfect, ${^{\perp}\cF_S}$ consists of 
the smooth algebraic groups $G$ such that $\pi_0(G) \in \cE_{S'}$.

\item[{\rm (iii)}] Returning to an arbitrary field $k$, the exact sequence
\[  0 \longrightarrow G^S \longrightarrow G \longrightarrow G_S
\longrightarrow 0 \]
(where $G^S$ is $S$-divisible and $G_S$ is $S$-primary torsion)
has a unique splitting in $\cC/\cF_S$, for any $G \in \cC$.
\end{enumerate}

\end{lemma}

\begin{proof} 
(i) Let $G \in \cC$. By \cite[Thm.~2.11]{Brion-II},
there exist two exact sequences in $\cC$
\[ 0 \longrightarrow H \longrightarrow G \longrightarrow U
\longrightarrow 0, 
\quad 
0 \longrightarrow M \longrightarrow H \longrightarrow A
\longrightarrow 0, \]
where $U$ is unipotent, $A$ an abelian variety, and $M$ 
of multiplicative type. In particular, $U \in \cC_S$.
Thus, $G \in \cC_S$ if and only if $A = 0$ and $M \in \cC_S$. 
The latter condition is equivalent to $M \in \cF_S$, in view of
the structure of groups of multiplicative type (see 
\cite[IV.1.3]{DG}). This shows the assertion on $\cC_S$.

Also, if $G \in \cC^S$, then $U = 0$ as $\cC^S$ is stable under
taking quotients. Let $T$ be the largest subtorus of $M$,
then we obtain an exact sequence in $\cC$
\[ 0 \longrightarrow M/T \longrightarrow G/T 
\longrightarrow A \longrightarrow 0, \]
where $M/T$ is finite and of multiplicative type.
In particular, $M/T$ is $n$-primary torsion for some positive integer
$n$, and hence the class of the above extension in 
$\Ext^1_{\cC}(A,M/T)$ is $n$-primary torsion as well. 
Thus, the pull-back of this extension under $n_A$ is trivial. 
Since $n_A$ is an epimorphism with finite kernel, we obtain
a map
\[ f: M/T \times A \longrightarrow G/T, \] 
which is an epimorphism with finite kernel as well.
As a consequence, there exists an exact sequence in $\cC$
\[ 0 \longrightarrow B \longrightarrow G/T 
\longrightarrow G_2 \longrightarrow 0, \]
where $B$ is an abelian variety (the image of $A$ under $f$)
and $G_2$ is finite. This yields in turn an exact sequence in $\cC$
\[ 0 \longrightarrow G_1 \longrightarrow G
\longrightarrow G_2 \longrightarrow 0, \]
where $G_1$ is a semi-abelian variety (extension of $B$ by
$T$). Using again the stability of $\cC^S$ under taking quotients,
it follows that $G_2 \in \cF_{S'} = \cE_{S'}$. Conversely, if 
$G_2 \in \cE_{S'}$, then $G_2 \in \cC^S$; also, $G_1 \in \cC^S$ 
since it is divisible. As $\cC^S$ is stable under extensions, 
we get that $G \in \cC^S$. This shows the assertion on $\cC^S$. 
That  on ${^{\perp}\cE_S}$ is obtained by similar arguments.

(ii)  Let $G \in {^{\perp}\cF_S}$. Since $k$ is perfect, the reduced 
subscheme $G_{\red}$ is a smooth subgroup of $G$ 
(see \cite[II.5.2.3]{DG}). Moreover, $G/G_{\red}$ is infinitesimal,
hence finite and $p$-primary torsion. As $p \in S$, it follows that 
$G = G_{\red}$ is smooth. Moreover, $\pi_0(G) \in \cE_{S'}$.
Conversely, every smooth connected algebraic group has no non-zero
finite quotient, and hence is an object of ${^{\perp}\cF_S}$.
Also, every finite $S'$-primary torsion group is in ${^{\perp}\cF_S}$. 
This yields the assertion by using again the stability of 
${^{\perp}\cF_S}$ under extensions.

(iii) Choose $n \in \langle S \rangle$ such that $n_{G_S} = 0$.
Then $n_{G^S}$ is an epimorphism, and its kernel is finite
for dimension reasons. So $n_{G^S}$ yields an isomorphism
in $\cC/\cF_S$. Thus, 
$\Hom_{\cC/\cF_S}(G_S,G^S) = 0 = \Ext^1_{\cC/\cF_S}(G_S,G^S)$.
This yields our statement. 
\end{proof}

\begin{remark}\label{rem:nonper}
If $k$ is imperfect, then ${^{\perp}\cF}$ contains non-smooth 
algebraic groups. Indeed, by \cite[Lem.~6.3]{Totaro}, there exists 
an exact sequence in $\cC$
\begin{equation}\label{eqn:hnt}
0 \longrightarrow \alpha_p \longrightarrow G 
\longrightarrow \bG_a \longrightarrow 0 
\end{equation}
such that every smooth connected subgroup scheme of $G$
is zero. In particular, the extension (\ref{eqn:hnt}) is nontrivial; 
also, $G$ is unipotent and non-smooth. 
If $f : G \to H$ is an epimorphism with $H \in \cF$, then $f$ induces 
an epimorphism $\bG_a \to H/f(\alpha_p)$. Since every finite
quotient of $\bG_a$ is zero, we see that $H$ is a quotient of 
$\alpha_p$, and hence must be zero by the nontriviality of 
the extension (\ref{eqn:hnt}). Thus, $G \in {^{\perp}\cF}$.
\end{remark}

We now obtain a description of the $S$-isogeny category, which
generalizes \cite[Prop.~3.6, Prop.~5.10]{Brion-II}:

\begin{proposition}\label{prop:Siso}

\begin{enumerate}

\item[{\rm (i)}] If $p \notin S$, then
$\cC/\cC_S = \cC/\cF_S = \cC/\cE_S$ is equivalent to the category
$\ucC^S$ with objects the algebraic groups $G$ such that
$\pi_0(G) \in \cE_{S'}$, and with morphisms given by
\[ \Hom_{\ucC^S}(G,H) = S^{-1} \Hom_{\cC}(G,H). \]

\item[{\rm (ii)}] If $p \in S$ and $k$ is perfect, then 
\[ \cC/\cF_S \cong \cU/(\cF_p \cap \cU) \times \cC/\cC_S. \]
Moreover, $\cC/\cC_S$ is equivalent to the category with objects
the extensions of finite \'etale $S'$-primary torsion groups by
semiabelian varieties, and with morphisms as in (i).

\end{enumerate}

\end{proposition}

\begin{proof}
(i) This follows from Lemmas \ref{lem:hom} and \ref{lem:pnotin}.

(ii) Let $G \in \cC$. We have a unique decomposition
$G = G^S \times G_S$ in $\cC/\cF_S$, as a consequence of
Lemma \ref{lem:pin} (iii). 
Moreover, we have a unique decomposition
$G_S = G_p \times G_{S \setminus \{p \} }$ in $\cC$ by 
(\ref{eqn:CFS}). As $G_{S \setminus \{ p \} } \in \cF_S$, it
follows that $G = G^S \times G_p$ in $\cC/\cF_S$. This yields
the assertion in view of Lemmas \ref{lem:ptors} (iii) and
\ref{lem:pin} (ii).
\end{proof}

\subsection{Homological dimension of $S$-isogeny categories}
\label{subsec:Shd}

We keep the notation of Subsection \ref{subsec:Siso}. For any 
prime number $\ell$, we denote by $\cd_{\ell}(k) = \cd_{\ell}(\Gamma)$ 
the $\ell$th cohomological dimension, i.e., the smallest integer
$n$ (if it exists) such that $H^i(\Gamma, M) = 0$ for all $i > n$ 
and all $\ell$-primary torsion, discrete $\Gamma$-modules $M$; 
if no such $n$ exists, then $\cd_{\ell}(k) := \infty$ 
(see \cite[6.1]{GS} for further developments on this notion). 
We may restrict to objects $M$ of $\Gamma \ldash \mod_{\ell}$ 
in the above definition, since every $\ell$-primary torsion, 
discrete $\Gamma$-module is the filtered direct limit of its finite 
submodules, and cohomology commutes with such limits.
We set $\cd_S(k) := \sup_{\ell \in S} (\cd_{\ell}(k))$ 
for any nonempty subset $S \subset \bP$; also,
$\cd_{\emptyset}(k) := 0$ and $\cd(k) := \cd_{\bP}(k)$.

\begin{theorem}\label{thm:Fhd}
Let $k$ be a perfect field of characteristic $p \geq 0$. 
Let $S$ be a subset of the set $\bP$ of prime numbers, and 
$S' := \bP \setminus S$. Then
\[ \hd(\cC/\cF_S) = 
\begin{cases} \cd_{S'}(k) + 1
& \text{if $p = 0$ or $p \in S$,} \\
\max(2, \cd_{S'}(k) + 1) 
& \text{otherwise.}
\end{cases} \] 
\end{theorem}

To prove this result, we will apply Theorem \ref{thm:hd}
to the artinian abelian category $\cA := \cC/\cF_S$ 
and to the set $\bP$ of all primes. 
By Lemma \ref{lem:quot}, we have equivalences of categories
$\cA_{\bP} \cong \cC_{\bP}/\cF_S$ and 
$\cA/\cA_{\bP} \cong \cC/\cC_{\bP}$; this yields
\begin{equation}\label{eqn:max}
\hd(\cC/\cF_S) = \max(\hd(\cC_{\bP}/\cF_S), \hd(\cC/\cC_{\bP})). 
\end{equation}
We will describe the categories $\cC_{\bP}/\cF_S$ and $\cC/\cC_{\bP}$ 
(Lemma \ref{lem:Ptors}), and determine their homological dimensions 
(Lemmas \ref{lem:hd} and \ref{lem:cd}). The assumption that $k$
is perfect will be used in the description of $\cC_{\bP}/\cF_S$ when
$p > 0$ and $p \notin S$ (Lemma \ref{lem:Ptors} (iii), based on
Lemma \ref{lem:ptors} (ii)).

In what follows, we set $\ucC := \cC/\cF$ (the full isogeny category 
of algebraic groups) and 
$\ucU := \cU/(\cF \cap \cU) = \cU(/\cF_p \cap \cU)$ 
(the isogeny category of unipotent algebraic groups). Also, we
denote by $\ucS$ the isogeny category of semi-abelian varieties, 
i.e., the full subcategory of $\ucC$ consisting of semi-abelian varieties.

\begin{lemma}\label{lem:Ptors}

\begin{enumerate}

\item[{\rm (i)}] If $p = 0$ then $\cA_{\bP} \cong \cE_{S'}$
and $\cA/\cA_{\bP} \cong \ucC$.

\item[{\rm (ii)}] If $p \in S$ then 
$\cA_{\bP} \cong \ucU \times \cE_{S'}$ and 
$\cA/\cA_{\bP} \cong \ucS$.

\item[{\rm (iii)}] If $p > 0$ and $p \notin S$ then 
$\cA_{\bP} \cong \cM_p \times \cU \times \cE_{S' \setminus \{ p \}}$ 
and $\cA/\cA_{\bP} \cong \ucS$.
\end{enumerate}

\end{lemma}

\begin{proof}
(i) Since $p = 0$, we have $\cC_{\bP} = \cF_{\bP} = \cE_{\bP}$,
hence $\cA_{\bP} \cong \cE_{S'}$ (in view of Lemma 
\ref{lem:quot}) and $\cA/\cA_{\bP} \cong \cC/\cF_{\bP}$. 

(ii) Using the isomorphism (\ref{eqn:CFS}), we obtain 
\[ \cA_{\bP} \cong \cC_{\bP}/\cF_S \cong \cC_p/\cF_p \times 
\cC_{\bP \setminus \{ p \} }/\cF_{S \setminus \{ p \} }. \]
Also, $\cC_{\bP \setminus \{ p \} } = \cE_{\bP \setminus \{ p \} }$,
hence 
$\cC_{\bP \setminus \{ p \} }/\cF_{S \setminus \{ p \} } 
\cong \cE_{S'}$. Moreover, 
$\cC_p/\cF_p \cong  \ucU$ by Lemma \ref{lem:ptors}.
This yields an equivalence 
$\cA_{\bP} \cong \ucU \times \cE_{S'}$. Furthermore,
$\cA/\cA_{\bP} \simeq \cC/\cC_{\bP}$ is equivalent to
$\ucS$ by Proposition \ref{prop:Siso}.

(iii) Using (\ref{eqn:CFS}) again, we obtain
$\cA_{\bP} \cong \cC_p \times \cE_{S' \setminus \{ p \} }$.
By Lemma \ref{lem:ptors}, it follows that 
$\cA_{\bP} \cong \cM_p \times \cU \times \cE_{S' \setminus \{ p \} }$.
Finally, the equivalence $\cA/\cA_{\bP} \cong \ucS$ is obtained 
as in (ii).
\end{proof}

\begin{lemma}\label{lem:hd}

\begin{enumerate}

\item[{\rm (i)}] $\hd(\ucU) = \hd(\ucC) = 1$.

\item[{\rm (ii)}] 
$ \hd(\ucS) = 
\begin{cases} 0 & \text{if $k$ is an algebraic extension of a finite field,} \\
1 & \text{otherwise.} 
\end{cases} $

\item[{\rm (iii)}] 
$\hd(\cM_{\ell}) = \hd(\cE_{\ell}) = \hd(\Gamma\ldash\mod_{\ell})$ 
for any prime number $\ell$.

\item[{\rm (iv)}] $\hd(\cU) = 2$.

\end{enumerate}

\end{lemma}

\begin{proof}
(i) The assertion on $\ucU$ is a consequence of results 
in \cite[V.3.6]{DG}, as explained in \cite[3.1.5]{Brion-III}. 
The assertion on $\ucC$ is the main result of \cite{Brion-II}.

(ii) Since $\ucS$ is a Serre subcategory of 
$\ucC$, we have $\hd(\ucS) \leq 1$ (see e.g. 
\cite[\S 3]{Oort64}). Moreover, $\hd(\ucS) = 0$ 
if and only if $k$ is an algebraic extension of a finite field, 
as follows from \cite[Prop.~5.8]{Brion-II}.

(iii) Just observe that $\cM_{\ell}$ is anti-equivalent to $\cE_{\ell}$
by Cartier duality (see \cite[IV.3.5]{DG}).

(iv) This follows from \cite[V.1.5.3, V.1.5.5]{DG}.
\end{proof}

Next, recall that
\[ \hd(\cE_S) = \sup_{\ell \in S} \big( \hd(\cE_{\ell}) \big)
= \sup_{\ell \in S} \big( \hd(\Gamma\ldash\mod_{\ell}) \big), \]
where the first equality follows from (\ref{eqn:hdsum}) and
(\ref{eqn:sums}). We now show:

\begin{lemma}\label{lem:cd}
$\hd(\Gamma\ldash\mod_{\ell}) = \cd_{\ell}(\Gamma) + 1$ for any prime
number $\ell$.
\end{lemma}

\begin{proof}
This follows from the results of \cite[Sec.~1]{Milne}, 
but we have been unable to understand some of the arguments 
given there. Thus, we provide an independent proof.

We begin with some observations on the category 
$\Gamma\ldash\mod_{\ell}$ of finite $\ell$-groups equipped with
a discrete action of $\Gamma$. This is an artinian and
noetherian category equipped with a forgetful functor
\[ \Gamma\ldash\mod_{\ell} \longrightarrow \mod_{\ell}, \quad
X \longmapsto \bar{X} \]
and with a duality
\[ X \longmapsto \Hom_{\Mod_{\ell}}(X,\bQ_{\ell}/\bZ_{\ell}), \]
where $\Gamma$ acts on the right-hand side via its given action 
on $X$ and the trivial action on $\bQ_{\ell}/\bZ_{\ell}$. 
We denote by $\Ind(\Gamma\ldash\mod_{\ell})$ the associated 
ind category; this is a locally noetherian category and its 
noetherian objects are exactly those of 
$\Gamma\ldash\mod_{\ell}$. By \cite[II.4.Thm.~1]{Gabriel},
it follows that $\Ind(\Gamma\ldash\mod_{\ell})$ is equivalent to
the category $\Gamma\ldash\Mod_{\ell}$ of discrete $\ell$-primary 
torsion $\Gamma$-modules (indeed, the latter category is locally 
noetherian, and its full subcategory consisting of noetherian objects
is $\Gamma\ldash\mod_{\ell}$ as well). Also, 
$\Ind(\Gamma\ldash\mod_{\ell})$ has enough injectives, and 
the natural map
\[ \Ext^i_{\Gamma\ldash\mod_{\ell}}(X,Y) \longrightarrow
\Ext^i_{\Ind(\Gamma\ldash\mod_{\ell})}(X,Y) \]
is an isomorphism for all $X,Y \in \Gamma\ldash\mod_{\ell}$
(as follows from the dual statement of (\ref{eqn:pro})).

We now adapt the argument of \cite[Prop.~p.~437]{Milne}. 
We claim that there exists a spectral sequence
\[ H^i(\Gamma,\Ext^j_{\Mod_{\ell}}(\bar{X},\bar{Y})) \Rightarrow
\Ext^{i + j}_{\Gamma\ldash\Mod_{\ell}}(X,Y) \]
for all $X \in \Gamma\ldash\mod_{\ell}$ and 
$Y \in \Gamma\ldash\Mod_{\ell}$. This can be proved by adapting
the argument of \cite[Thm.~I.0.3]{Milne-II}, which yields a 
similar spectral sequence in the setting of discrete 
$\Gamma$-modules; alternatively, the claim can be derived from
\cite[10.8.7]{Weibel} and the subsequent discussion. We provide
a direct argument: first note that
\[ H^0(\Gamma, \Hom_{\Mod_{\ell}}(\bar{X},\bar{Y})) 
= \Hom_{\Gamma\ldash\Mod_{\ell}}(X,Y). \]
Also, the functor $H^0(\Gamma, ?)$ is left exact and the forgetful 
functor is exact. So the claim will follow from the spectral sequence 
of a composite functor (see \cite[5.8.3]{Weibel}), once we show
that the endofunctor $Y \mapsto \Hom_{\Mod_{\ell}}(\bar{X},\bar{Y})$
of $\Gamma\ldash\Mod_{\ell}$ takes injectives to $\Gamma$-acyclics.

Let $I$ be an injective object of $\Gamma\ldash\Mod_{\ell}$. Then the map
\[ \iota: I \longrightarrow \Hom(\Gamma,I), \quad
x \mapsto (g \mapsto g \cdot x) \]
is injective and $\Gamma$-equivariant, where $\Gamma$ acts on 
the right-hand side via right multiplication on itself. Moreover,
the image of $\iota$ is contained in the subgroup
\[ \Hom_{\cont}(\Gamma,I) := \lim_{\to} \big( \Hom(\Gamma/U,I^U) \big), \]
where the direct limit is over all open normal subgroups $U$ of
$\Gamma$. This subgroup is an object of $\Gamma\ldash\Mod_{\ell}$, and is
injective in that category: indeed,
\[ \Hom_{\Gamma\ldash\Mod_{\ell}}(Z,\Hom_{\cont}(\Gamma,I)) \cong
\Hom^{\Gamma}_{\cont}(Z \times \Gamma,I) \cong 
\Hom_{\Mod_{\ell}}(\bar{Z},\bar{I}) \]
for any $Z \in \Gamma\ldash\Mod_{\ell}$, and $\bar{I}$ is injective
in $\Mod_{\ell}$ by (i). Thus, $\iota$ identifies $I$ with a
summand of $\Hom(\Gamma,I)$. The latter is $\Gamma$-acyclic in
view of \cite[6.3.3, 6.11.13]{Weibel}. This completes the proof 
of the claim.

As $\Ext^j_{\Mod_{\ell}}(?,?) = 0$ for all $j \geq 2$, this claim
yields the inequality 
\[ \hd(\Gamma\ldash\mod_{\ell}) \leq \cd_{\ell}(\Gamma) + 1. \]
To show the opposite inequality, consider first the case where 
$\cd_{\ell}(\Gamma) =: n$ is finite. We may then find an
$X \in \Gamma\ldash\mod_{\ell}$ such that $H^n(\Gamma,X) \neq 0$ and
$X$ is killed by $\ell$. We have 
$X = \Hom_{\Mod_{\ell}}(Y,\bQ_{\ell}/\bZ_{\ell})$
for a unique $Y \in \Gamma\ldash\mod_{\ell}$; then $Y$ is killed by
$\ell$ as well. By the claim and the fact that 
$\bQ_{\ell}/\bZ_{\ell}$ is injective in $\Mod_{\ell}$, 
there is an injection
\[ H^n(\Gamma,X) \hookrightarrow 
\Ext^n_{\Gamma\ldash\Mod_{\ell}}(Y,\bQ_{\ell}/\bZ_{\ell}). \]
Also, the exact sequence 
\[ 0 \longrightarrow \bZ/\ell\bZ \longrightarrow 
\bQ_{\ell}/\bZ_{\ell} \stackrel{\ell}{\longrightarrow}
\bQ_{\ell}/\bZ_{\ell}\longrightarrow 0 \]
yields an injection 
\[ \Ext^n_{\Gamma\ldash\Mod_{\ell}}(Y,\bQ_{\ell}/\bZ_{\ell})
\hookrightarrow
\Ext^{n+1}_{\Gamma\ldash\Mod_{\ell}}(Y,\bZ/\ell \bZ). \]
This completes the proof in this case. In the case where 
$\cd_{\ell}(\Gamma)$ is infinite, we apply the above arguments 
to an infinite sequence of positive integers $n$ such that
$H^n(\Gamma, ?) \neq 0$.
\end{proof}

We may now complete the proof of Theorem \ref{thm:Fhd}
by freely using the above three lemmas. 

If $p = 0$, then 
$\hd(\cA_{\bP}) = \hd(\cE_{S'}) = \cd_{S'}(k) + 1$
and $\hd(\cA/\cA_{\bP}) = \hd(\ucC) = 1$. 

If $p \in S$, then 
$\hd(\cA_{\bP}) = \max(\hd(\ucU), \hd(\cE_{S'})) = \cd_{S'}(k) + 1$
and $\hd(\cA/\cA_{\bP}) = \hd(\ucS) \leq 1$.

Finally, if $p >0$ and $p \notin S$, then 
\[ \hd(\cA_{\bP}) 
= \max(\hd(\cU), \hd(\cM_p), \hd(\cE_{S' \setminus \{ p \}})) 
= \max(2, \cd_{S'}(k) + 1) \]
(where the second equality follows from the fact that
$\hd(\cM_p) = \cd_p(k) \leq 1$, see \cite[Prop.~6.1.9]{GS})
and $\hd(\cA/\cA_{\bP}) = \hd(\ucS) \leq 1$. 

In either case, we conclude by using (\ref{eqn:max}).

\begin{remark}\label{rem:Fhd}
(i) Still assuming $k$ perfect, the homological dimension of 
the subcategory $\cL$ of $\cC$ may be determined along similar lines. 
Indeed, $\cL$ contains the torsion subcategory $\cC_{\bP}$, 
as follows from Lemma \ref{lem:ptors}. Thus, 
for any Serre subcategory $\cB$ of $\cC_{\bP}$,
we have $(\cL/\cB)_{\tors} = \cC_{\bP}/\cB$ in view of 
Lemma \ref{lem:quot}. Using Theorem \ref{thm:hd}, it follows that 
\begin{equation}\label{eqn:LB} 
\hd(\cL/\cB) = \max(\hd(\cC_{\bP}/\cB), \hd(\cL/\cC_{\bP})) 
\end{equation}
Also, $\hd(\cL/\cC_{\bP}) \leq 1$ as $\cL/\cC_{\bP}$ is 
a Serre subcategory of $\cC/\cC_{\bP}$. Taking $\cB = \cF_S$
and using the fact that $\hd(\cC_{\bP}/\cF_S) \geq 1$
(Lemmas \ref{lem:Ptors} and \ref{lem:hd}), we obtain 
\[ \hd(\cL/\cF_S) = \hd(\cC/\cF_S) \]
under the assumptions of Theorem \ref{thm:Fhd}.

Alternatively, one may determine directly $\hd(\cL/\cF_S)$ 
by using the equivalence of categories (\ref{eqn:MU}). 

\smallskip

\noindent
(ii) When $k$ is imperfect, the homological dimension of $\cC$ seems 
to be unknown, already in the case where $k$ is separably closed.
Under that assumption, it is asserted in \cite[p.~75]{Kraft}
that the category $\Ac_k$ of (commutative) affine $k$-group
schemes has homological dimension $3$. Since $\Ac_k$ is equivalent
to $\Pro(\cL)$ (see \cite[V.2.2.2, V.2.3.1]{DG}), this yields 
$\hd(\cL) = 3$ in view of Lemma \ref{lem:hdproj}. Using the methods 
developed above, one can deduce from this that $\hd(\cC) = 3$ as well.
But the proof of the equality $\hd(\Ac_k) = 3$, left to the reader 
as an exercise (``Aufgabe''), could not be reconstituted so far. 
\end{remark}

\subsection{The radicial isogeny category}
\label{subsec:radicial}

We keep the notation of Subsections \ref{subsec:Siso} and
\ref{subsec:Shd}, and assume that $p > 0$. Let $\cI$ be 
the full subcategory of $\cC$ with objects the infinitesimal 
group schemes; then $\cI$ is a Serre subcategory of $\cC_p$, 
as follows from \cite[IV.3.5.3]{DG}. We say that $\cC/\cI$ 
is the \emph{radicial isogeny category}.

By \cite[Lem.~2.2]{Brion-II}, the pair $(\cC,\cI)$ satisfies the
lifting property. Also, recall that an algebraic group
is infinitesimal if and only if it is finite and connected.
In view of the exact sequence (\ref{eqn:neutral}), it follows
that every finite group scheme $G$ lies in a unique exact
sequence
\begin{equation}\label{eqn:inf}
0 \longrightarrow I \longrightarrow G \longrightarrow E 
\longrightarrow 0, 
\end{equation}
where $I := G^0$ is infinitesimal, and $E := \pi_0(G)$ is finite
\'etale; moreover, $\Hom_{\cF}(I,E) = 0 = \Hom_{\cF}(E,I)$. 
When $k$ is perfect, the extension (\ref{eqn:inf}) has a unique
splitting (as follows from \cite[II.5.2.4]{DG}). Thus, for perfect $k$,
the product map $\cI \times \cE \to \cF$ is an equivalence
of categories (see \cite[IV.3.5]{DG} for further developments).

\begin{lemma}\label{lem:smooth}
If $G \in \cC$ is smooth, then $G \in {^{\perp}\cI}$. The converse
holds when $k$ is perfect.
\end{lemma}

\begin{proof}
If $G$ is smooth, then so is every quotient of $G$.
In particular, every infinitesimal quotient is trivial, 
that is, $G \in {^{\perp}\cI}$. 

Conversely, assume that $k$ is perfect and let $G \in {^{\perp}\cI}$.
Recall that the reduced subscheme $G_{\red}$ is a smooth subgroup 
of $G$ and $G/G_{\red}$ is infinitesimal. It follows that $G$ 
is smooth.
\end{proof}

(If $k$ is imperfect, then ${^{\perp}\cI}$ contains non-smooth
algebraic groups, as follows from Remark \ref{rem:nonper}.)

We now obtain a more explicit description of the category
$\cC/\cI$, over an arbitrary field $k$ of characteristic $p > 0$. 
To state it, we recall some standard results on Frobenius kernels. 
For any algebraic group $G$ and any positive integer $n$, 
we denote by $F^n_{G/k} : G \to G^{(p^n)}$ the $n$th iterated 
relative Frobenius morphism. Then the $n$th Frobenius
kernel $G_n := \Ker(F^n_{G/k})$ is an infinitesimal subgroup 
scheme of $G$; moreover, $G/G_n$ is smooth for $n \gg 0$, 
and every infinitesimal subgroup scheme of $G$ is contained 
in some $G_n$ (see \cite[VIIA.8]{SGA3}). This implies readily:

\begin{proposition}\label{prop:Iiso}
With the above notation, $\cC/\cI$ is equivalent to its full subcategory
with objects the smooth algebraic groups. Moreover, we have
\[ \Hom_{\cC/\cI}(G,H) = \lim_{\to} \big( \Hom_{\cC}(G, H/H_n) \big) \]
for all smooth algebraic groups $G,H$, where the direct limit is taken 
over all positive integers.
\end{proposition}

Next, we determine the homological dimension of the radicial
isogeny category over a perfect field:

\begin{theorem}\label{thm:Ihd}
Let $k$ be a perfect field of characteristic $p > 0$. Then
\[ \hd(\cC/\cI) = \max(2,\cd(k) + 1). \]
\end{theorem}

\begin{proof}
We argue as in the proof of Theorem \ref{thm:Fhd}. Let 
$\cA := \cC/\cI$ and consider the set $\bP$ of all primes.
Then $\cA$ is an artinian abelian category and
$\cA/\cA_{\bP} \cong \cC/\cC_{\bP} \cong \ucS$
by Lemmas \ref{lem:quot} and \ref{lem:Ptors}. Thus,
$\hd(\cA/\cA_{\bP}) \leq 1$ (Lemma \ref{lem:hd}). Also, 
$\cA_{\bP} \cong \cC_{\bP}/\cI \cong 
(\cC_p \times \cE_{p'})/\cI$
by Lemma \ref{lem:quot} again combined with (\ref{eqn:CFS}). 
Moreover, the decomposition $\cC_p = \cM_p \times \cU$
(Lemma \ref{lem:ptors}) induces a decomposition
$\cI = \cM_p \times \cJ$,
where $\cJ$ denotes the (Serre) subcategory of $\cU$
with objects the infinitesimal unipotent groups. Thus,
$\cA_{\bP} \cong \cU/\cJ \times \cE_{p'}$ and hence
\[ \hd(\cA_{\bP}) = 
\max(\hd(\cU/\cJ), \cd_{p'}(k) + 1) \]
by (\ref{eqn:hdsum}) and Lemma \ref{lem:cd}. Also, recall 
that $\cd_p(k) \leq 1$ by \cite[Prop.~6.1.9]{GS}. 
Thus, it suffices to show that 
\[ \hd(\cU/\cJ) = 2 \] 
in view of Theorem \ref{thm:hd}.

We have $\hd(\cU/\cJ) \leq \hd(\cU)$ by Proposition \ref{prop:hdmax},
since the pair $(\cU,\cJ)$ satisfies the lifting property.
This yields $\hd(\cU/\cJ) \leq 2$ in view of Lemma \ref{lem:hd}.
So it suffices in turn to check that
\begin{equation}\label{eqn:k} 
\Ext^2_{\cU/\cJ}(\bG_a, \nu_p) = k, 
\end{equation}
where $\bG_a$ denotes the additive group, and $\nu_p$
denotes the constant group scheme associated with $\bZ/p\bZ$
(as in \cite[I.2.15]{Oort66}).

The short exact sequence
\[ 0 \longrightarrow \nu_p \longrightarrow \bG_a
\stackrel{F - \id}{\longrightarrow} \bG_a \longrightarrow 0 \]
yields a long exact sequence
\[ \Ext^1_{\cU/\cJ}(\bG_a, \bG_a) 
\stackrel{F_* - \id}{\longrightarrow}
\Ext^1_{\cU/\cJ}(\bG_a, \bG_a) \to
\Ext^2_{\cU/\cJ}(\bG_a, \nu_p) \to
\Ext^2_{\cU/\cJ}(\bG_a, \bG_a), \]
where $F_*$ denotes the push-out. As $\bG_a/I \cong \bG_a$
for any infinitesimal subgroup $I \subset \bG_a$, Lemma
\ref{lem:long} implies that the canonical map
\[ Q^i : \Ext^i_{\cU}(\bG_a, \bG_a) \longrightarrow
\Ext^i_{\cU/\cJ}(\bG_a, \bG_a) \]
is surjective for any $i \geq 0$. 

We now claim that $\Ext^2_{\cU}(\bG_a, \bG_a) = 0$. 
This is a direct consequence of results of \cite{DG}.
More specifically, the pro category $\Pro(\cC)$ is equivalent 
with that of commutative affine group schemes (see 
\cite[V.2.2.c]{DG}) and this induces an equivalence
of $\Pro(\cU)$ with the category of commutative unipotent
group schemes (as follows from Lemma \ref{lem:prob} and
\cite[IV.2.2.5]{DG}). The latter category is anti-equivalent
to that of $V$-primary torsion modules over the Dieudonn\'e
ring $\bD$, where $V$ denotes the Verschiebung;
this anti-equivalence takes a commutative unipotent
group scheme $G$ to its Dieudonn\'e module $M(G)$, and
$\bG_a$ to $\bD/\bD V$ (see \cite[V.1.4.2, V.1.4.3]{DG}).
By combining Lemma \ref{lem:hdproj} and \cite[V.1.5.1]{DG},
this leads to isomorphisms
\[ \Ext^i_{\cU}(G,H) \stackrel{\cong}{\longrightarrow} 
\Ext^i_{\bD \ldash \Mod}(M(H),M(G)) \]
for all $G,H \in \cU$ and all $i \geq 0$, where 
$\bD \ldash \Mod$ denotes the category of all left $\bD$-modules.
In view of the projective resolution 
\[ 0 \longrightarrow \bD \stackrel{V}{\longrightarrow}
\bD \longrightarrow \bD/\bD V \longrightarrow 0 \]
in $\bD \ldash \Mod$, we thus obtain that 
$\Ext^i_{\cU}(G,\bG_a) = 0$ for all $G \in \cU$
and $i \geq 2$. This proves the claim.

By this claim and the surjectivity of $Q^2$, we have 
$\Ext^2_{\cU/\cJ}(\bG_a, \bG_a) = 0$ and hence
\begin{equation}\label{eqn:coker} 
\Ext^2_{\cU/\cJ}(\bG_a,\nu_p) \cong \Coker(F_* - \id). 
\end{equation}

We now claim that $Q^1$ is injective.
Let $\xi \in \Ext^1_{\cU}(\bG_a, \bG_a)$ 
be represented by an exact sequence in $\cU$
\[ 0 \longrightarrow \bG_a 
\stackrel{\alpha}{\longrightarrow} G 
\stackrel{\beta}{\longrightarrow} \bG_a \longrightarrow 0, \]
which splits in $\cU/\cJ$. Then we may choose 
$\gamma \in \Hom_{\cU/\cJ}(G,\bG_a)$ such that
$\gamma \circ \alpha = \id$ in  $\End_{\cU/\cJ}(\bG_a)$.
In view of Lemma \ref{lem:homq}, $\gamma$ is represented by
$\delta \in \Hom_{\cU}(G,\bG_a/I)$ for some infinitesimal
subgroup $I \subset \bG_a$. Then $I \subset \Ker(F^n)$
for some $n \geq 1$; thus, we may assume that 
$I = \Ker(F^n)$, and identify the quotient map 
$\bG_a \to \bG_a/I$ with $F^n : \bG_a \to \bG_a$. 
Then $\delta \circ \alpha - F^n \in \End_{\cU}(\bG_a)$
represents zero in $\End_{\cU/\cJ}(\bG_a)$. As $\bG_a$
is smooth, it follows from Lemma \ref{lem:equiv}
that $\delta \circ \alpha - F^n = 0$ in $\End_{\cU}(\bG_a)$.
This yields a commutative diagram of exact sequences in $\cU$
\[ \xymatrix{
0 \ar[r] & \bG_a \ar[r]^{\alpha} \ar[d]_{F^n} & 
G \ar[r]^{\beta} \ar[d] \ar[dl]_{\delta} & \bG_a \ar[r] \ar[d]_{\id} & 0 \\
0 \ar[r] & \bG_a \ar[r] & G/I \ar[r] & \bG_a \ar[r] & 0.
}
\]
Thus, the bottom exact sequence splits, i.e. $F^n_*(\xi) = 0$
in $\Ext^1_{\cU}(\bG_a,\bG_a)$. But $\Ext^1_{\cU}(\bG_a,\bG_a)$
is a free module over $\End_{\cU}(\bG_a)$ (acting either on
the left or on the right) in view of \cite[V.1.5.2]{DG}. Also,
recall that $\End_{\cU}(\bG_a)$ is the noncommutative 
polynomial ring $k[F]$ (see \cite[II.3.4.4]{DG}).
It follows that $\xi = 0$, proving the claim.

By that claim, $Q^1$ is an isomorphism. In view of (\ref{eqn:coker}), 
it follows that $\Ext^2_{\cU/\cJ}(\bG_a,\nu_p)$ is the cokernel of the map
\[ F_* - \id : \Ext^1_{\cU}(\bG_a, \bG_a) \longrightarrow
\Ext^1_{\cU}(\bG_a, \bG_a). \]
This yields the isomorphism (\ref{eqn:k}) by using 
the structure of $\Ext^1_{\cU}(\bG_a,\bG_a)$ recalled above.
\end{proof}
 
By Theorems \ref{thm:Fhd} and \ref{thm:Ihd}, we have 
$\hd(\cC/\cI) = \hd(\cC)$ if $k$ is perfect. Under that assumption,
we now determine the homological dimension of the \'etale isogeny 
category $\cC/\cE_S$. If $p \notin S$, then $\cE_S = \cF_S$ and 
Theorem \ref{thm:Fhd} applies again. Thus, we may assume that 
$p \in S$:

\begin{theorem}\label{thm:Ehd}
Let $k$ be a perfect field of characteristic $p > 0$, and $S$ 
a set of prime numbers containing $p$. Then
\[ \hd(\cC/\cE_S) = \max(2, \cd_{S'}(k) + 1). \]
\end{theorem}

\begin{proof}
We argue again as in the proof of Theorem \ref{thm:Fhd}; we omit 
the details. Let $\cA := \cC/\cE_S$, then 
$\cA/\cA_{\bP} \cong \cC/\cC_{\bP} \cong \ucS$ 
has homological dimension $1$, and 
\[ \cA_{\bP} \cong \cC_p/\cE_p \times \cE_{S'}
\cong \cU/\cE_p \times \cM_p \times \cE_{S'} \]
has homological dimension 
\[ \max(\hd(\cU/\cE_p), \cd_p(k), \cd_{S'}(k) + 1). \]
Since $\cd_p(k) \leq 1$, it suffices to show that 
$\hd(\cU/\cE_p) = 2$. By Lemma \ref{lem:lifting}, 
the pair $(\cU,\cE_p)$ satisfies the lifting property.
Using Proposition \ref{prop:hdmax} and Lemma \ref{lem:hd}, 
it follows that $\hd(\cU/\cE_p) \leq \hd(\cU) = 2$.
Also, by adapting the proof of Theorem \ref{thm:Ihd}, 
one may check that 
\[ \Ext^2_{\cU/\cE_p}(\bG_a, \alpha_p) = k, \] 
where $\alpha_p$ lies in the exact sequence
\[ 0 \longrightarrow \alpha_p \longrightarrow \bG_a 
\stackrel{F}{\longrightarrow} \bG_a \longrightarrow 0. \]
More specifically, this exact sequence yields
an exact sequence
\[ \Coker(F_*) \longrightarrow
\Ext^2_{\cU/\cE_p}(\bG_a, \alpha_p) \longrightarrow
\Ext^2_{\cU/\cE_p}(\bG_a,\bG_a), \]
where $F_*$ is viewed as an endomorphism of
$\Ext^1_{\cU/\cE_p}(\bG_a,\bG_a)$. Moreover,
the canonical map 
$Q^i: \Ext^i_{\cU}(\bG_a,\bG_a) \to 
\Ext^i_{\cU/\cE_p}(\bG_a,\bG_a)$
is still surjective for any $i \geq 0$. As
$\Ext^2_{\cU}(\bG_a,\bG_a) = 0$, it follows that
$\Ext^2_{\cU/\cE_p}(\bG_a, \alpha_p) \cong 
\Coker(F_*)$. 
Also, $Q^1$ is injective as seen by the argument of 
Theorem \ref{thm:Ihd}, where the infinitesimal subgroup 
$I \subset \bG_a$ is replaced by a finite \'etale
subgroup $E$, and $F^n$ by an endomorphism
of $\bG_a$ with kernel $E$. It follows that
$\Ext^2_{\cU/\cE_p}(\bG_a, \alpha_p)$ is
isomorphic to the cokernel of $F_*$ viewed
as an endomorphism of $\Ext^1_{\cU}(\bG_a,\bG_a)$.
We conclude by using again the structure of that extension 
group as a module over $\End_{\cU}(\bG_a) \cong k[F]$.
\end{proof}

\begin{remark}\label{rem:Ihd}
By combining Remark \ref{rem:Fhd} (i) with Theorems \ref{thm:Ihd}
and \ref{thm:Ehd}, one can show that
\[ \hd(\cL/\cB) = \hd(\cF/\cB) \]
when $\cB$ equals $\cI$ or $\cE_S$.
\end{remark}

\subsection{Invariance under purely inseparable field extensions}
\label{subsec:pi}

Consider a field extension $k'/k$. For any algebraic $k$-group $G$,
we obtain an algebraic $k'$-group $G_{k'} = G \otimes_k k'$ by base
change. This yields a functor
\[ \otimes_k \, k' : \cC = \cC_k \longrightarrow \cC_{k'}, \]
which is exact and faithful.

\begin{theorem}\label{thm:pi}
Let $k'/k$ be a purely inseparable field extension, and $S$ a set 
of prime numbers containing $p := \charac(k)$. Then the above 
base change functor induces equivalences of categories
\[ \cC_k/\cF_{S,k} \stackrel{\cong}{\longrightarrow}  
\cC_{k'}/\cF_{S,k'} \quad \text{and} \quad  
\cC_k/\cI_k \stackrel{\cong}{\longrightarrow} \cC_{k'}/\cI_{k'}. \]
\end{theorem}

\begin{proof}
We adapt the argument of \cite[Thm.~3.11]{Brion-II}. Denote by 
\[ Q' : \cC_{k'} \longrightarrow \cC_{k'}/\cF_{S,k'} \] 
the quotient functor. Then the composite functor
\[ \cC_k  \stackrel{\otimes_k \, k'}{\longrightarrow}
\cC_{k'}  \stackrel{Q'}{\longrightarrow} 
\cC_{k'}/\cF_{S,k'} \]
is exact and takes every object of $\cF_{S,k}$ to zero. Thus, this 
functor factors uniquely through an exact functor
\[ \cC_k/\cF_{S,k} \longrightarrow \cC_{k'}/\cF_{S,k'} \]
that we still denote by $\otimes_k \, k'$. We will check that
$\otimes_k \, k'$ is essentially surjective, faithful, and full.

Let $G' \in \cC_{k'}$. By \cite[Lem.~3.10]{Brion-II}, there exist
a $G \in \cC_k$ and an epimorphism $f: G' \to G_{k'}$ with 
kernel in $\cI_k$, and hence in $\cF_{S,k}$ since 
$p \in S$. Thus, $\otimes_k \, k'$ is essentially surjective.

Next, let $G,H \in \cC_k$ and 
$\varphi \in \Hom_{\cC_k/\cF_{S,k}}(G,H)$ such that $\varphi_{k'}$ 
is zero. By Lemmas \ref{lem:homq} and \ref{lem:lifting},
we have $\varphi = Q(f)$ for some $f \in \Hom_{\cC_k}(G,H/H')$,
where $H' \subset H$ and $H' \in \cF_{S,k}$. Then
$\Im(f_{k'}) \in \cF_{S,k'}$ in view of 
\cite[III.1.Lem.~2]{Gabriel}. It follows that 
$\Im(f) \in \cF_{S,k}$ and hence that $\varphi = 0$. So
$\otimes_k \, k'$ is faithful.

To show that this functor is full, it suffices to check that every
$\varphi' \in \Hom_{\cC_{k'}/\cF_{S,k'}}(G_{k'},H_{k'})$
can be written as $\varphi_{k'}$ for some $\varphi$ as above.
By Lemmas \ref{lem:homq} and \ref{lem:lifting} again, we have
$\varphi' = Q'(f')$ for some
$f' \in \Hom_{\cC_{k'}}(G_{k'},H_{k'}/H')$, where 
$H' \subset H_{k'}$ and $H_{k'}/H' \in \cF_{S,k'}$. Applying 
\cite[Lem.~3.10]{Brion-II} once more, we obtain a $k$-subgroup 
$K \subset H$ such that $H' \subset K_{k'}$ and $K_{k'}/H' \in \cI_{k'}$.
Then $K_{k'} \in \cF_{S,k'}$ and hence $K \in \cF_{S,k}$.
Thus, we may replace $H'$ with $K_{k'}$, and assume that
$f' \in \Hom_{\cC_{k'}}(G_{k'},H_{k'}/K_{k'})$. Replacing
$H$ with $H/K$ yields a further reduction to the case where 
$f' \in \Hom_{\cC_{k'}}(G_{k'},H_{k'})$.  

Consider the graph $\Gamma(f')$, a $k'$-subgroup of 
$G_{k'} \times H_{k'}$.
By \cite[Lem.~3.10]{Brion-II} again, there exists a 
$k$-subgroup $\Delta \subset G \times H$ such that
$\Gamma(f') \subset \Delta_{k'}$ and 
$\Delta_{k'}/\Gamma(f') \in \cI_{k'}$. Since the projection
$\Gamma(f') \to G_{k'}$ is an isomorphism in $\cC_{k'}$, 
we see that the projection $\Delta_{k'} \to G_{k'}$ is 
an epimorphism in that category, with kernel 
$\Delta_{k'} \cap H_{k'} \in \cI_{k'}$. Thus, the projection
$\Delta \to G$ is an epimorphism in $\cC_k$, with kernel
$\Delta \cap H \in \cI_k$. Now let $\Gamma$ be the image of
$\Delta$ under 
\[ \id \times q : G \times H \longrightarrow 
G \times H/(H \cap \Delta), \]
where $q$ denotes the quotient. Then the projection
$\Gamma \to G$ is an isomorphism, i.e., $\Gamma$ is the
graph of some $f \in \Hom_{\cC_k}(G,H/ (\Delta\cap H))$. 
As $\Gamma(f') \subset \Delta_{k'}$, we obtain
$f_{k'} = q_{k'} \circ f'$. So 
$\varphi' = Q'(f_{k'}) = Q(f)_{k'}$ as desired.

This completes the proof for the $S$-isogeny category
$\cC/\cF_S$. The argument for the radicial isogeny category
$\cC/\cI$ is entirely similar.
\end{proof}

Our main theorem now follows readily by combining Theorems 
\ref{thm:Fhd}, \ref{thm:Ihd}, \ref{thm:Ehd} and \ref{thm:pi}.

\begin{remark}\label{rem:pi}
With the assumptions of Theorem \ref{thm:pi}, the base change functor
also induces equivalences of categories
\[ \cL_k/\cF_{S,k} \stackrel{\cong}{\longrightarrow}  
\cL_{k'}/\cF_{S,k'} \quad \text{and} \quad  
\cL_k/\cI_k \stackrel{\cong}{\longrightarrow} \cL_{k'}/\cI_{k'}. \]
as follows by the same argument.

Taking for $k'$ the perfect closure of $k$ and using
the equivalence of categories 
$\cL_{k'} \cong \cU_{k'} \times \cM_{k'}$ (\ref{eqn:MU}), 
this yields further equivalences (with our original notation
$\cL = \cL_k$, $\cI = \cI_k$, $\ldots$)
\[ \cL/\cF_S \cong \cU/(\cF_p \cap U) \times \cM/\cM_S, 
\quad \cL/\cI \cong \cU/(\cI \cap \cU) \times \cM/\cM_p. \]
The former equivalence also follows from Proposition 
\ref{prop:Siso}. 
\end{remark}

\medskip

\noindent
{\bf Acknowledgements.}
I thank Christian Ausoni, Baptiste Calm\`es, Mathieu Florence, 
Philippe Gille, Hanspeter Kraft, Tam\'as Szamuely and Bruno Vallette 
for stimulating discussions or email exchanges. Special thanks are 
due to St\'ephane Guillermou for his decisive help with results of 
Section \ref{sec:prel} (most importantly, Lemmas \ref{lem:projective}
and \ref{lem:ext}) and also for his careful reading and valuable comments.
Also, many thanks to the anonymous referees for their thorough 
and helpful reports.

\bibliographystyle{amsalpha}

\end{document}